\documentclass{article}
\usepackage[noadjust,nocompress]{cite}
\usepackage{amssymb}
\usepackage{verbatim}
\usepackage{amsmath,stmaryrd}
\usepackage{amsthm,mathtools}
\usepackage{abstract}
\usepackage{tikz,tikz-cd}
\usepackage{indentfirst}
\usepackage{chngcntr}
\usepackage{hyperref}
\hypersetup{hidelinks}
\usepackage{pifont}
\usepackage{geometry}
\usepackage{setspace,xparse}

\numberwithin{equation}{section}
\newcommand{\rmnum}[1]{\romannumeral#1}
\newcommand{\Rmnum}[1]{\uppercase\expandafter{\romannumeral#1}}
\newtheoremstyle{hjx}
{6pt}
{6pt}
{\itshape}
{}
{\bfseries}
{}
{1em}
{\thmname{#1}\,\,\thmnumber{#2}\,\,\thmnote{#3}}
\theoremstyle{hjx}
\newtheorem{theorem}{\textbf{Theorem}}[section]
\newtheorem{proposition}{\textbf{Proposition}}[section]
\newtheorem{lemma}{\textbf{Lemma}}[section]
\allowdisplaybreaks[4]
\newtheoremstyle{hjxx}
{10pt}
{10pt}
{}
{}
{\bfseries}
{}
{1em}
{\thmname{#1}\,\,\thmnumber{#2}\,\,\thmnote{#3}}
\theoremstyle{hjxx}

\NewDocumentCommand{\llrrlr}{O{1}O{n}}{
	\settowidth{\hangindent}{2em}
	{#1}\hangafter=1
	{#2}
}
\NewDocumentCommand{\llrrind}{O{0pt}O{m}O{m}}{%
  \settowidth{\labelwidth}{#2}
  \addtolength{\labelwidth}{#1}
  \noindent
  \makebox[\labelwidth][l]{#2}
  \hangindent=\labelwidth
  \hangafter=1
  #3\par
}
\begin{document}

\title{\scshape\bf Shen-Larsson modules over the Lie algebras of divergence zero vector fields on \texorpdfstring{$\mathbb{C}^n$}{C^n}}
\author{Jinxin Hu$^{1}$ and Rencai L\"u$^{2}$}
\footnotetext[1]{J. Hu, Department of Mathematics, Soochow University, Suzhou 215506, China,
{\em E-mail address}: \url{20244007004@stu.suda.edu.cn}}
\footnotetext[2]{R. L\"u, Department of Mathematics, Soochow University, Suzhou 215506, China,
{\em E-mail address}: \url{rlu@suda.edu.cn}}
\date{}
\maketitle

\begin{center}
    \vspace*{6pt}
    \begin{minipage}{0.9\textwidth}
		\begin{center}
		\textbf{Abstract}
		\end{center}
	Let $n\geq 2$ be an integer, 
	$S_n$ be the Lie algebra of vector fields on $\mathbb{C}^n$
	with zero divergence, and $D_n$ be the Weyl algebra
	over the polynomial algebra $A_n=\mathbb{C}[t_1,t_2,\cdots,t_n]$.
	In this paper, we study the simplicity of the Shen-Larsson $S_n$-module $F(P,M)$,
	where $P$ is a simple $D_n$-module and $M$ is a simple $\mathfrak{sl}_n$-module.
	We obtain the necessary and sufficient conditions for $F(P,M)$
	to be an irreducible module, and determine
	all simple subquotients of $F(P,M)$ when it is reducible.

	\llrrind[1em][\textbf{Keywords}][Lie algebra of divergence zero vector fields; simple module; Shen-Larsson module; Weyl algebra
	]

	\vspace*{6pt}

	\textbf{2020 MSC\quad}17B10, 17B65, 17B66

\end{minipage}
\vspace*{12pt}
\end{center}

\section{Introduction}
We denote by $\mathbb{Z}$, $\mathbb{Z}_+$, $\mathbb{Z}_-$
and $\mathbb{C}$ the sets of all integers, nonnegative integers,
nonpositive integers and complex numbers, respectively.
For any positive integer $n$, let $A_n$ be the polynomial algebra
$\mathbb{C}[t_1,t_2,\cdots,t_n]$ and 
$\mathcal{A}_n$ be the Laurent polynomial algebra
$\mathbb{C}[t_1^{\pm1},t_2^{\pm1},\cdots,t_n^{\pm1}]$.

The study of Lie algebras of vector fields began with the fundamental work of S. Lie and E. Cartan in
the late 19th century and the early 20th century.
Lie algebras of this type include the Lie algebras
\(W_n\) consisting of the derivations of \(A_n\) and the Lie algebras \(\mathcal{W}_n\) 
consisting of the derivations of \(\mathcal{A}_n\).
A. N. Rudakov \cite{R1,R2} obtained the first classification results concerning representations of
$W_n$. 
Rudakov's results tackle the classification of a class of irreducible \(W_n\)-representations 
satisfying some natural topological conditions.
These modules constitute a special class of tensor modules, also known as Shen-Larsson modules.

The Shen-Larsson modules were introduced by Shen and Larsson; see \cite{SG,L}.
In 1986, Shen \cite{SG} constructed a Lie algebra monomorphism from 
$W_n$ (resp. $\mathcal{W}_n$)
to the semidirect product Lie algebras 
$W_n\ltimes \mathfrak{gl}_n(A_n)$
(resp. $\mathcal{W}_n\ltimes \mathfrak{gl}_n(\mathcal{A}_n)$),
which are special full toroidal Lie algebras.
Here $\mathfrak{gl}_n(A_n)=A_n\otimes_{\mathbb{C}}\mathfrak{gl}_n$
and $\mathfrak{gl}_n(\mathcal{A}_n)=\mathcal{A}_n\otimes_{\mathbb{C}}\mathfrak{gl}_n$.
We denote by $D_n$ (resp. $\mathcal{D}_n$) the Weyl algebra
over $A_n$ (resp. $\mathcal{A}_n$).
For any module $P$ over $D_n$
(resp. $\mathcal{D}_n$)
and any module $M$ over the general linear Lie algebra
$\mathfrak{gl}_n$, the tensor product 
$P\otimes_{\mathbb{C}} M$
becomes a $W_n$-module (resp. $\mathcal{W}_n$-module) via Shen's monomorphism.
The modules defined above are called the Shen-Larsson modules (over $W_n$ and $\mathcal{W}_n$, respectively) and are denoted by $F(P,M)$.

The Shen-Larsson modules of $W_1$ and their extensions were studied extensively in the 1970s and
the 1980s by B. Feigin, D. Fuks, I. Gelfand, and others; see, for example, \cite{F,FF}.
G. Liu, R. Lü and K. Zhao established the necessary and 
sufficient conditions for Shen-Larsson modules to be irreducible 
over $W_n$ and $\mathcal{W}_n$; furthermore, they determined 
all submodules of Shen-Larsson modules when these modules are reducible; see \cite{LLZ}.
For more related results, we refer readers to \cite{XL,B,BF1,BF2,CD,E,E1,S} and the references therein.

The Shen-Larsson modules of $W_n$ and $\mathcal{W}_n$ play an important role in the classification of simple Harish-Chandra
modules (weight modules with finite-dimensional weight spaces) over these Lie algebras.
In 2016, Y. Billig and V. Futorny \cite{BF1} obtained the complete classification of simple Harish-Chandra modules
of $\mathcal{W}_n$, where the Shen-Larsson modules over $\mathcal{W}_n$ are important components in their classification
results. 
The classification of simple bounded modules (i.e., weight modules for which the 
dimensions of their weight spaces are uniformly bounded by some constant) of $W_n$
was completed in \cite{XL}. The result in \cite{XL} states that every simple
bounded module is a Shen-Larsson module or a submodule of a Shen-Larsson module.
D. Grantcharov and V. Serganova \cite{GS} obtained the complete classification of simple Harish-Chandra modules
of $W_n$. They showed that every nontrivial simple Harish-Chandra $W_n$-module
is either a Shen-Larsson module with finite weight multiplicities
or the unique simple submodule of such a module arising in the de Rham complex.

Let $W_{m,n}:=\mathrm{Der}(A_m\otimes\land(n))$ and $\mathcal{W}_{m,n}:=\mathrm{Der}(\mathcal{A}_m\otimes\land(n))$ 
be the Witt superalgebras,
where $\land(n)$ is the exterior algebra in $n$ odd
variables. The simplicity of Shen-Larsson modules over $W_{m,n}$ and $\mathcal{W}_{m,n}$
was studied in \cite{XW}. In \cite{XL1}, the simple bounded modules over $W_{m,n}$
were classified. Every such module is a simple quotient of a Shen-Larsson module.
In \cite{XL2,BFIK}, the simple strong Harish-Chandra modules over $\mathcal{W}_{m,n}$
were classified. Every such module is either a simple quotient of a Shen-Larsson module
or a module of highest weight type.

Let $\mathcal{S}_n$ $(n\geq 2)$ be the Lie algebra of divergence zero
vector fields on an $n$-dimensional
torus with respect to degree derivations.
The Shen-Larsson modules of $\mathcal{S}_n$ were studied in \cite{LGW,T}
and their simplicity was determined in \cite{DGYZ}.
The simple Harish-Chandra modules over the Virasoro-like algebra
(which is the universal central extension of $\mathcal{S}_2$)
were studied and partially classified in \cite{LTW,LT}.
Let $\bar{S}_n$ $(n\geq 2)$ (resp. $S_n$ $(n\geq 2)$) be the Lie algebra of vector fields on $\mathbb{C}^n$ 
with constant (resp. zero) divergence. 
We classified the 
simple Harish-Chandra modules of $\bar{S}_2$ in \cite{HR}.
Any such module over $\bar{S}_2$ is a Shen-Larsson module
or a simple subquotient of a Shen-Larsson module.
Recently, V. Futorny and S. Tantubay \cite{FT}
considered the Shen-Larsson functor from the category of
modules for the symplectic Lie algebra to the category of modules
for the Hamiltonian Lie algebra and obtained a number of simple cuspidal modules over the Hamiltonian Lie algebra.

In this paper, we establish the necessary and sufficient conditions 
for the Shen-Larsson modules of $S_n$ to be irreducible, and determine 
all their simple subquotients when these modules are reducible.
The paper is arranged as follows. In Section 2, we collect some basic notation
and results for later use.
In Section 3, we study the simplicity of the Shen-Larsson $S_n$-module $F(P,M)$,
where $P$ is a simple $D_n$-module and $M$ is a simple $\mathfrak{sl}_n$-module.
We establish Theorems \ref{t1} and \ref{t2}, which together constitute the 
main results of this paper. 
Theorem \ref{t1} shows that the Shen-Larsson $S_n$-module $F(P,M)$
is simple provided that 
$M$ is not isomorphic to any fundamental module.
Theorem \ref{t2} addresses the remaining cases.
In Section 4, 
we apply the main results to the Shen-Larsson modules
$F(P,M)$, where both $P$ and $M$ are weight modules,
and obtain all their simple subquotients explicitly.

\section{Notation and preliminaries}
In this section, we collect some notation and results for later use.
Let $e_i\in\mathbb{Z}^n$ be the $n$-tuple with $1$ in the $i$-th 
component and $0$ in all other components.
For any $\alpha\in \mathbb{Z}^n$, let $\alpha_i$ be the $i$-th component
of $\alpha$.
For any $\alpha,\beta\in \mathbb{Z}^n$, we write $\alpha\geq\beta$
if $\alpha_i\geq\beta_i$ for all $i=1,2,\cdots,n$.
A module $M$ over a Lie algebra $\mathfrak{g}$ is called trivial if 
$\mathfrak{g}M=0$.
For any Lie algebra $\mathfrak{g}$, we denote by $U(\mathfrak{g})$ the universal enveloping algebra
of $\mathfrak{g}$. 

Recall that the Lie algebra $\mathcal{W}_n = \sum_{i=1}^{n}\mathcal{A}_n\partial_i$ is equipped with the following Lie bracket:
\[
\left[ \sum_{i=1}^n f_i\partial_i, \sum_{j=1}^n g_j\partial_j \right] = \sum_{i,j=1}^n \left( f_j\partial_j(g_i) - g_j\partial_j(f_i) \right) \partial_i,
\]
where $f_i, g_j \in \mathcal{A}_n$ and we set 
$\partial_i = \frac{\partial}{\partial t_i}$ for convenience.
Moreover, $W_n = \sum_{i=1}^{n} A_n\partial_i$ forms a Lie subalgebra of $\mathcal{W}_n$.
For $n \geq 2$, let $S_n \subset W_n$ denote the Lie subalgebra consisting of all derivations with zero divergence, namely,
\[
S_n = \left\{ \sum_{i=1}^n p_i\partial_i \;\bigg|\; p_i \in A_n, \ \sum_{i=1}^n \partial_i(p_i)=0 \right\}.
\]

The Weyl algebra $\mathcal{D}_n$ (resp. $D_n$) is the unital associative algebra 
generated by elements 
$$\{t_i^{\pm1},\partial_i|1\leq i\leq n\}\quad
(\mathrm{resp.}\quad \{t_i,\partial_i|1\leq i\leq n\})$$
subject to the relations:
$[t_i,t_j]=[\partial_i,\partial_j]=0$ and $[\partial_i,t_j]=\delta_{ij}$ for all
$i,j=1,2,\cdots,n$.

Let $d_i:=t_i\partial_i$ for all $1\leq i\leq n$ and 
$\mathfrak{G}$ be the associative algebra $D_n$
or any Lie subalgebra of $W_n$ that contains $d_1,d_2,\cdots,d_n$. A $\mathfrak{G}$-module 
$V$ is called a weight module if the action of $d_1,d_2,\cdots,d_n$ on $V$ is
diagonalizable, i.e., $V=\bigoplus_{\lambda \in \mathbb{C} ^n}^{}{V_{\lambda}}$, where
$$V_{\lambda}=\left\{ v\in V \middle| d_iv=\lambda _iv\,\,,\,\,i=1,2,\cdots,n \right\}.$$
$V_{\lambda}$ is called the weight space with weight $\lambda$. We set
$\mathrm{supp}(V):=\left\{ \lambda \in \mathbb{C} ^n \middle| V_{\lambda}\ne 0 \right\}$.

Let $f:\mathfrak{G}_1\xrightarrow{}\mathfrak{G}_2$ be a homomorphism of Lie
algebras or associative algebras and $V$ be a $\mathfrak{G}_2$-module.
We can make $V$ into a $\mathfrak{G}_1$-module by $x\cdot v=f(x)v$, 
$\forall x\in\mathfrak{G}_1$, $v\in V$. The resulting module is denoted by $V^f$.

The (full) Fourier transform $F$ 
is the automorphism of $D_n$ defined by $F(t_i)=\partial_i$, $F(\partial_i)=-t_i$
for $i=1,2,\cdots, n$. 
Let $D_{(i)}$ be the subalgebra of $D_n$ generated by $t_i$ and $\partial_i$,
and let $F_{(i)}=F|_{D_{(i)}}$ be the restriction of $F$ to $D_{(i)}$.
Note that $D_n\cong D_{(1)}\otimes D_{(2)}\otimes\cdots\otimes D_{(n)}$.
We recall the classification of simple weight $D_n$-modules.
\begin{lemma}[(\cite{FGM})]\label{l2}
	(\rmnum{1})  Any simple weight $D_{(i)}$-module is isomorphic to
one of the following simple weight $D_{(i)}$-modules:
$$t_{i}^{\lambda_i}\mathbb{C}[t_i^{\pm1}]\,\,,\,\,
A_{(i)}:=\mathbb{C}[t_i]\,\,,\,\,A_{(i)}^{F_{(i)}}(\cong\mathbb{C}[t_i^{\pm1}]/\mathbb{C}[t_i]),$$
where $\lambda_i\in\mathbb{C}\backslash\mathbb{Z}$.

(\rmnum{2})  Let $P$ be any simple weight $D_n$-module. Then $P\cong V_1\otimes V_2\otimes
\cdots\otimes V_n$,
where $V_i$ is a simple $D_{(i)}$-module.
Therefore, the support set of any simple weight $D_n$-module
is of the form $X=X_1\times X_2\times\cdots\times X_n$, where
$X_i\in \{a +\mathbb{Z}, \mathbb{Z}_+, \mathbb{Z}_{<0}\}$, 
$a\in \mathbb{C}\backslash\mathbb{Z}$. 
\end{lemma}

We denote by $E_{ij}$ the $n\times n$
square matrix with $1$ as its $(i, j)$-entry and $0$ in all other entries. We have
the general linear Lie algebra
$$\mathfrak{gl}_n=\bigoplus_{1\leq i,j\leq n}\mathbb{C}E_{ij}$$
and the special linear Lie algebra $\mathfrak{sl}_n$ that
consists of all $n\times n$ matrices with zero trace.
Let 
$$\mathfrak{H}=\mathrm{span}\left\{ E_{ii} \middle| 1\le i\le n \right\}
\quad\mathrm{and}\quad \mathfrak{h}=\mathrm{span}\left\{ h_i \middle| 1\le i\le n-1 \right\},$$
where $h_i=E_{ii}-E_{i+1,i+1}$.
These are Cartan subalgebras of $\mathfrak{gl}_n$ and $\mathfrak{sl}_n$, respectively.
Let $$\Lambda^{+}=\left\{ \lambda \in \mathfrak{h} ^* \middle| 
\lambda \left( h_i \right) \in \mathbb{Z} _+ ,\forall 1\le i\leq n-1\right\}$$
be the set of dominant weights with respect to $\mathfrak{h}$.
An $\mathfrak{sl}_n$-module $V$ is called a weight module if the action of $\mathfrak{h}$
on $V$ is diagonalizable, i.e., $V=\oplus_{\lambda\in\mathfrak{h}^*}V_{\lambda}$,
where $V_\lambda=\left\{ v\in V \middle| hv=\lambda \left( h \right) v \,\,,\,\,\forall
h\in\mathfrak{h}\right\}$ is called the weight space of $V$ with the weight $\lambda$.
Denote by $\mathrm{supp}(V)=\left\{ \lambda \in \mathfrak{h} ^*\middle| V_{\lambda}\ne 0 \right\} $
the support set of $V$.
For any $\psi\in \mathfrak{h}^*$,
let $V(\psi)$ be the simple $\mathfrak{sl}_n$-module with highest weight $\psi$.

We make $V(\psi)$ into a $\mathfrak{gl}_n$-module
$V(\psi,b)$ by letting the identity matrix $I$ act as multiplication by a scalar
$b\in \mathbb{C}$. 
Define the fundamental weights $\delta_i\in\mathfrak{h}^*$
by $\delta_i(h_j)=\delta_{ij}$ for all $i,j=1, 2,\cdots, n-1$.
For convenience, we set $\delta_0=\delta_n=0\in\mathfrak{h}^*$. It is well known 
that the fundamental $\mathfrak{gl}_n$-modules $V(\delta_k,k)$,
$k = 0, 1,\cdots, n$, can be realized as 
the exterior product $\bigwedge\nolimits_{}^k{\left( \mathbb{C} ^{n\times 1} \right)}
$ with the action given by
$$X\left( v_1\land v_2\land \cdots \land v_k \right) =\sum_{i=1}^k{v_1\land 
\cdots \land v_{i-1}\land Xv_i\land v_{i+1}\land \cdots \land v_k},$$
where $X\in \mathfrak{gl}_n$.

Write $t^{\alpha}=t_1^{\alpha_1}t_2^{\alpha_2}\cdots t_n^{\alpha_n}$
for any $\alpha\in \mathbb{Z}^n$ and 
$\partial^{\alpha}
=\partial_1^{\alpha_1}\partial_2^{\alpha_2}\cdots \partial_n^{\alpha_n}$
for any $\alpha\in \mathbb{Z}_+^n$.
We recall the definition of the Shen-Larsson modules of $W_n$.
Shen's Lie algebra homomorphism $\iota:W_n\rightarrow D_n\otimes
U(\mathfrak{gl}_n)$ is defined by
\begin{equation}\label{eq16}
	\iota(t^\alpha\partial_i)=t^\alpha\partial_i\otimes 1+
\sum_{s=1}^{n}\partial_s(t^\alpha)\otimes E_{si},
\end{equation}
for all $\alpha\in\mathbb{Z}_+^n$ and $i=1,2,\cdots,n$.
This map extends uniquely to an associative algebra homomorphism from
$U(W_n)$ to $D_n\otimes U(\mathfrak{gl}_n)$,
which we also denote by $\iota$.
The Shen-Larsson modules of \(W_n\) are defined 
to be the \(W_n\)-modules constructed via tensor products, denoted
by \(F(P,M) := (P\otimes_{\mathbb{C}} M)^{\iota}\), where $P$ is an arbitrary $D_n$-module and
$M$ is an arbitrary $\mathfrak{gl}_n$-module.

We denote by $\varepsilon_i\in\mathbb{C}^{n\times 1}$ the 
column vector with $1$ in the 
$i$-th entry and $0$ elsewhere.
Let $P$ be a simple $D_n$-module. The $W_n$-modules $F(P,V(\delta_k,k))$ for 
$0\le k\le n$ are generalizations of the modules of differential $k$-forms. 
These
modules form the de Rham complex
$$
	0\xrightarrow{}F\left( P,V\left( \delta _0,0 \right) \right) 
	\xrightarrow{\pi _0}F\left( P,V\left( \delta _1,1 \right) \right) 
	\xrightarrow{\pi _1}F\left( P,V\left( \delta _2,2 \right) \right) 
	\xrightarrow{}\cdots \xrightarrow{\pi _{n-1}}
	F\left( P,V\left( \delta _n,n \right) \right) \xrightarrow{}0,
$$
where
\begin{align*}
	\pi_k:F\left( P,V\left( \delta _k,k \right) \right)
&\xrightarrow{}F\left( P,V\left( \delta _{k+1},k+1\right) \right),\\
p\otimes v&\xrightarrow{}\sum_{l=1}^n{ \partial _lp  
\otimes \varepsilon_l\land v },
\end{align*}
for all $p\in P$, $v\in V\left( \delta _k,k \right)$,
$k=0,1,\cdots,n-1$, 
see \cite[Lemma 3.2]{LLZ}. For $1\le r\le n$, let $$L_n(P,r):=\pi_{r-1}(F(P,V(\delta_{r-1},r-1)))$$
and set $L_n(P,0)=0$. By the definition of $\pi_{r-1}$, $L_n(P,r)$ is spanned by
$$\sum_{k=1}^n{ \partial _kp  \otimes \left( \varepsilon _k\land 
\varepsilon _{i_2}\land \cdots \land \varepsilon _{i_r} \right)}=
\sum_{k=1}^n{ \partial _kp\otimes E_{kj}v}
,$$
where $p\in P$ and $j$ is chosen so that 
$v=\varepsilon _j\land \varepsilon _{i_2}\land \cdots \land \varepsilon _{i_r}\ne 0$.

Let $$\widetilde{L_n}\left( P,r \right):=
\left\{ v\in F\left( P,V\left( \delta _r,r \right) \right) \middle| W_nv
\subseteq L_n\left( P,r \right) \right\}.$$ 
Both ${L_n}\left( P,r \right)$
and $\widetilde{L_n}\left( P,r \right)$ are $W_n$-submodules of $F(P,V(\delta_r,r))$.
It is clear that $\widetilde{L_n}\left( P,r \right)/{L_n}\left( P,r \right)$
is trivial.
Recall the following results for $L_n(P,r)$ and $\widetilde{L_n}(P,r)$ from
\cite[Corollary 3.3, Theorem 3.5]{LLZ}.

\begin{lemma}[{(\cite{LLZ})}]\label{l430}
	Let $P$ be a simple $D_n$-module.

	(a) $\widetilde{L_n}\left( P,r \right)=\mathrm{Ker}(\pi_r)$ for all $r=0,1,\cdots,n-1$.
	
	(b) $L_n(P,r)$ is a proper $W_n$-submodule of $F(P,V(\delta_r,r))$ for 
	all $r=1,\cdots,n-1$.

	(c) For every $r=1,\cdots,n-1$, $F(P,V(\delta_r,r))$
	is not simple as a $W_n$-module.
\end{lemma}
\section{Shen-Larsson modules of \texorpdfstring{$S_n$}{S\_n}}
We introduce some notation.
For any $\alpha\in\mathbb{Z}^n$ and $i,j=1,2,\cdots,n$,
let
$$L_{ij}^{\alpha}:=t^{\alpha}\left( \left( 1+\alpha _j \right) d_i-\left( 1+\alpha _i \right) d_j \right) 
\in\mathcal{W}_n.$$
Note that $L_{ij}^{\alpha}\in S_n$ if $\alpha\geq -e_i-e_j$.
The algebra $S_n$ is spanned by 
$$\{L_{ij}^{\alpha}|i,j=1,2,\cdots,n;i\ne j;
\alpha\in\mathbb{Z}^n;\alpha\geq -e_i-e_j\}.$$
For any $i,j=1,2,\cdots,n$ with $i\ne j$ and $\alpha\geq-e_i-e_j$,
we have 
\begin{align*}
\iota \left( L_{ij}^{\alpha} \right) =&L_{ij}^{\alpha}\otimes 1+\left( 1+\alpha _i \right) \left( 1+\alpha _j \right) t^{\alpha}\otimes \left( E_{ii}-E_{jj} \right) 
\\
&+\left( 1+\alpha _j \right) \sum_{s\ne i}^{}{\alpha _st^{\alpha +e_i-e_s}\otimes E_{si}}
-\left( 1+\alpha _i \right) \sum_{s\ne j}^{}{\alpha _st^{\alpha +e_j-e_s}\otimes E_{sj}},
\end{align*}
which implies that $\iota(S_n)\subseteq D_n\otimes U(\mathfrak{sl}_n)$.
Hence, the restriction map $\iota|_{S_n}$ is a Lie algebra homomorphism from $S_n$ to $D_n\otimes U(\mathfrak{sl}_n)$.
For an arbitrary $D_n$-module $P$ and
an arbitrary $\mathfrak{sl}_n$-module $M$, the Shen-Larsson module of $S_n$ is defined to be
the tensor module $(P\otimes_{\mathbb{C}} M)^{\iota|_{S_n}}$.
In fact, the Shen-Larsson modules of $S_n$ are exactly the modules obtained by restriction from the Shen-Larsson modules of $W_n$.
We keep the notation $F(P,M)$ for the Shen-Larsson modules of $S_n$.

\subsection{Simplicity of Shen-Larsson modules}

We need the following lemma.
\begin{lemma}\label{l5}
	For any $\beta\in\mathbb{Z}_+^n$ and $1\leq i,j\leq n$ with $i\ne j$,
	we have $t^{\beta}\otimes \left( E_{ij} \right) ^2\in \iota(U(S_n))$.
\end{lemma}
\begin{proof}
	Equation (\ref{eq16}) defines a Lie algebra homomorphism from $\mathcal{W}_n$
	to $\mathcal{D}_n\otimes U(\mathfrak{gl}_n)$ upon extending the domain of
	$\alpha$ to $\mathbb{Z}^n$. We denote this homomorphism and its induced
	associative algebra homomorphism on $U(\mathcal{W}_n)$ by $\hat{\iota}$.
	Note that $\iota=\hat{\iota}|_{W_n}$.

	For any $\alpha\in \mathbb{Z}^n$, $m\in\mathbb{Z}$ and $i,j=1,2,\cdots,n$ with $i\ne j$, we have 
\begin{align*}
	&\hat{\iota} \left( L_{ij}^{\alpha-me_i}
		\right) \cdot \hat{\iota} \left( t^{me_i}\partial _j \right) 
\\
=&\left( 1+\alpha _j \right) \hat{\iota} \left( t^{\alpha -\left( m-1 \right) e_i}
\partial _i \right) \cdot \hat{\iota} \left( t^{me_i}\partial _j \right) 
-\left( 1+\alpha _i-m \right) \hat{\iota} \left( t^{\alpha -me_i+e_j}\partial _j \right) 
\cdot \hat{\iota} \left( t^{me_i}\partial _j \right) 
\\
=&\left( 1+\alpha _j \right) 
\left( t^{\alpha -\left( m-1 \right) e_i}\partial _i\otimes 1
+\sum_{s=1}^n{\left( \alpha _s-\delta _{si}m+\delta _{si} \right) 
t^{\alpha -\left( m-1 \right) e_i-e_s}\otimes E_{si}} \right) \\
&\phantom{\left( 1+\alpha _j \right)}\cdot \left( t^{me_i}\partial _j\otimes 1+mt^{\left( m-1 \right) e_i}\otimes E_{ij} \right) 
\\
&-\left( 1+\alpha _i-m \right) \left( t^{\alpha -me_i+e_j}
\partial _j\otimes 1+\sum_{s=1}^n{\left( \alpha _s-\delta _{si}m+
\delta _{sj} \right) t^{\alpha -me_i+e_j-e_s}\otimes E_{sj}} \right) \\
&\phantom{-\left( 1+\alpha _i-m \right)}\cdot \left( t^{me_i}\partial _j\otimes 1+mt^{\left( m-1 \right) e_i}\otimes E_{ij} \right) 
\\
=&\left( 1+\alpha _j \right) t^{\alpha -\left( m-1 \right) e_i}\left( t^{me_i}\partial _i+mt^{\left( m-1 \right) e_i} \right) \partial _j\otimes 1
\\
&+m\left( 1+\alpha _j \right) t^{\alpha -\left( m-1 \right) e_i}\left( t^{\left( m-1 \right) e_i}\partial _i+\left( m-1 \right) t^{\left( m-2 \right) e_i} \right) \otimes E_{ij}
\\
&+\left( 1+\alpha _j \right) \sum_{s=1}^n{\left( \alpha _s-\delta _{si}m+\delta _{si} \right) t^{\alpha +e_i-e_s}\partial _j\otimes E_{si}}
\\
&+m\left( 1+\alpha _j \right) \sum_{s=1}^n{\left( \alpha _s-\delta _{si}m+\delta _{si} \right) t^{\alpha -e_s}\otimes E_{si}E_{ij}}
\\
&-\left( 1+\alpha _i-m \right) t^{\alpha +e_j}\partial _j\partial _j\otimes 1
\\
&-m\left( 1+\alpha _i-m \right) t^{\alpha -e_i+e_j}\partial _j\otimes E_{ij}
\\
&-\left( 1+\alpha _i-m \right) \sum_{s=1}^n{\left( \alpha _s-\delta _{si}m+\delta _{sj} \right) t^{\alpha +e_j-e_s}\partial _j\otimes E_{sj}}
\\
&-m\left( 1+\alpha _i-m \right) \sum_{s=1}^n{\left( \alpha _s-\delta _{si}m+\delta _{sj} \right) t^{\alpha -e_i+e_j-e_s}\otimes E_{sj}E_{ij}}
\\
=&\left( 1+\alpha _j \right) t^{\alpha +e_i}\partial _i\partial _j\otimes 1+m\left( 1+\alpha _j \right) t^{\alpha}\partial _j\otimes 1
\\
&+m\left( 1+\alpha _j \right) t^{\alpha}\partial _i\otimes E_{ij}+m\left( m-1 \right) \left( 1+\alpha _j \right) t^{\alpha -e_i}\otimes E_{ij}
\\
&+\left( 1+\alpha _j \right) \sum_{s=1}^n{\left( \alpha _s-\delta _{si}m+\delta _{si} \right) t^{\alpha +e_i-e_s}\partial _j\otimes E_{si}}
\\
&+m\left( 1+\alpha _j \right) \sum_{s=1}^n{\left( \alpha _s-\delta _{si}m+\delta _{si} \right) t^{\alpha -e_s}\otimes E_{si}E_{ij}}
\\
&-\left( 1+\alpha _i-m \right) t^{\alpha +e_j}\partial _j\partial _j\otimes 1-m\left( 1+\alpha _i-m \right) t^{\alpha -e_i+e_j}\partial _j\otimes E_{ij}
\\
&-\left( 1+\alpha _i-m \right) \sum_{s=1}^n{\left( \alpha _s-\delta _{si}m+\delta _{sj} \right) t^{\alpha +e_j-e_s}\partial _j\otimes E_{sj}}
\\
&-m\left( 1+\alpha _i-m \right) \sum_{s=1}^n{\left( \alpha _s-\delta _{si}m+\delta _{sj} \right) t^{\alpha -e_i+e_j-e_s}\otimes E_{sj}E_{ij}}
.
\end{align*}

Then we can write
\begin{equation}\label{eq7}
	\hat{\iota} \left( 
			L_{ij}^{\alpha-me_i}\right) \cdot \hat{\iota} \left( t^{me_i}\partial _j \right) 
=-m^3\left( t^{\alpha -2e_i+e_j}\otimes \left( E_{ij} \right) ^2 \right) +m^2u_2+mu_1+u_0,
\end{equation}
where $u_2,u_1,u_0\in \mathcal{D}_n\otimes U(\mathfrak{gl}_n)$ are independent of $m$.
Setting $m=0,1,2,3$ in (\ref{eq7}), we get a linear system of equations
whose coefficient matrix is nonsingular.
Then we obtain that
\begin{equation}\label{eq8}
	\begin{aligned}
		t^{\alpha +e_j-2e_i}\otimes \left( E_{ij} \right) ^2=&
		-\frac{1}{6}\hat{\iota} \left( L_{ij}^{\alpha -3e_i} \right) 
		\cdot \hat{\iota} \left( t^{3e_i}\partial _j \right) +\frac{1}{2}
		\hat{\iota} \left( L_{ij}^{\alpha -2e_i} \right) \cdot
		\hat{\iota} \left( t^{2e_i}\partial _j \right) 
\\
&-\frac{1}{2}\hat{\iota} \left( L_{ij}^{\alpha -e_i} \right) \cdot 
\hat{\iota} \left( t^{e_i}\partial _j \right) +\frac{1}{6}\hat{\iota} 
\left( L_{ij}^{\alpha} \right) \cdot \hat{\iota} \left( \partial _j \right) 
.
	\end{aligned}
\end{equation}

Note that if $\alpha\geq 2e_i-e_j$, then
$L_{ij}^{\alpha-me_i}$ and $t^{me_i}\partial_j$ belong to $S_n$
for $m=0,1,2,3$.
Their products therefore belong to $U(S_n)$.
Since $\hat{\iota}$ agrees with $\iota$ on $S_n$, each product
on the right-hand side of (\ref{eq8}) lies in $\iota(U(S_n))$.
Taking $\alpha=\beta+2e_i-e_j$, we obtain
\[
t^\beta\otimes (E_{ij})^2\in\iota(U(S_n))
\qquad\text{for all }\beta\in\mathbb{Z}_+^n.
\]
\end{proof}

Now we can give the first main result in this section.
	\begin{theorem}\label{t1}
	Let $P$ be a simple $D_n$-module and
$M$ be a simple $\mathfrak{sl}_n$-module such that $M$ is not isomorphic to
$V(\delta_k)$ for any $k=0,1,\cdots,n$.
Then $F(P,M)$ is a simple $S_n$-module.
\end{theorem}
\begin{proof}
    Suppose that $V$ is a nonzero proper submodule of $F(P,M)$.
	Let $\sum_{k=1}^q{p_k\otimes v_k}$ be a nonzero element in $V$.

	\textbf{Claim 1.}\quad For any $u\in D_n$ and $i,j=1,2,\cdots,n$ with $i\ne j$,
	we have $\sum_{k=1}^q{up_k\otimes \left( E_{ij} \right) ^2v_k}\in V$.
	
	Since $\iota(\partial_s)=\partial_s\otimes 1$ for all $1\leq s\leq n$,
	we have $\sum_{k=1}^q{\partial_s p_k\otimes v_k}\in V$.
	Hence, we have 
	$$\sum_{k=1}^q{\partial^\alpha p_k\otimes v_k}\in V,$$
	for all $\alpha\in\mathbb{Z}_+^n$. By Lemma \ref{l5}, we obtain that
	$$\sum_{k=1}^q{t^\beta\partial^\alpha p_k\otimes \left( E_{ij} \right) ^2v_k}\in V,$$
	for all $\alpha,\beta\in\mathbb{Z}_+^n$.
	Now Claim 1 follows from the fact that the algebra $D_n$ is generated 
	by $t_r,\partial_s$ with $1\leq r,s\leq n$.

\textbf{Claim 2.}\quad Assume that $p_1,p_2,\cdots,p_q$ are linearly independent. Then for
any $k=1,2,\cdots,q$ and $i,j=1,2,\cdots,n$ with $i\ne j$, we have $(E_{ij})^2v_k=0$.

Since $P$ is an irreducible $D_n$-module, by the density theorem in ring theory,
for any $p\in P$ and any $k=1,2,\cdots,q$, there exists some $u(p,k)\in D_n$ such that
$u(p,k)p_k=p$ and $u(p,k)p_l=0$ for $l\ne k$. Then from Claim 1, we see that
$P\otimes (E_{ij})^2v_k\subseteq V$ for all
$k=1,2,\cdots,q$.

Set $M_1:=\left\{ v\in M \middle| P\otimes v\subseteq V \right\}$.
Let $v\in M_1$. For any $p\in P$ and $r,s=1,2,\cdots,n$ with $r\ne s$, we have
$$p\otimes E_{rs}v=\left( t_r\partial _s \right) \cdot\left( p\otimes v \right) 
-t_r\partial _sp\otimes v\in V.$$
We see that $M_1$ is an $\mathfrak{sl}_n$-submodule of $M$, and thus it must be $0$ or $M$.
Since $V$ is a proper submodule of $F(P,M)$, we must have $M_1=0$.
Claim 2 follows.

From now on we assume that $p_1,p_2,\cdots,p_q$ are linearly independent.

\textbf{Claim 3.}\quad For any $i,j=1,2,\cdots,n$ with $i\ne j$, we have 
$(E_{ij})^2M=0$.

Let $s,r=1,2,\cdots,n$ with $s\ne r$. We have
$$\left( t_s\partial _r \right) \cdot \left( \sum_{k=1}^q{p_k\otimes v_k} \right) =
\sum_{k=1}^q{\left( t_s\partial _rp_k\otimes v_k+p_k\otimes E_{sr}v_k \right)}\in V
.$$
By Claim 1, for any $u\in D_n$, we have
$$\sum_{k=1}^q{ut_s\partial _rp_k\otimes \left( E_{ij} \right) ^2v_k}
+\sum_{k=1}^q{up_k\otimes \left( E_{ij} \right) ^2E_{sr}v_k}\in V.$$
By Claim 2, we have
$$\sum_{k=1}^q{up_k\otimes \left( E_{ij} \right) ^2E_{sr}v_k}\in V.$$
Since $p_1,p_2,\cdots,p_q$ are linearly independent, by taking different values of $u$ in the above
formula, we deduce that
$$P\otimes \left( E_{ij} \right) ^2E_{sr}v_k\subseteq V,$$
for all $k=1,2,\cdots,q$.
This means that
$\left( E_{ij} \right) ^2E_{sr}v_k\in M_1$ for any $k=1,2,\cdots,q$. Since $M_1=0$, 
we have
$\left( E_{ij} \right) ^2E_{sr}v_k=0$ for any $k=1,2,\cdots,q$.
Repeating this procedure, we deduce that
$$\left( E_{ij} \right) ^2U(\mathfrak{sl}_n)v_k=0,$$
for all $k=1,2,\cdots,q$.
Since $M$ is an irreducible $\mathfrak{sl}_n$-module, we obtain that
$(E_{ij})^2M=0$.
Claim 3 follows.

By Claim 3, $(E_{ij})^2M=0$ for all $1\leq i,j\leq n$
with $i\ne j$.
Thus every root vector of $\mathfrak{sl}_n$ acts locally
nilpotently on $M$.
Recall from \cite[Lemma 2.3(2)]{LZ1} that a simple
$\mathfrak{sl}_n$-module, not necessarily a weight module,
is finite-dimensional if and only if every root vector
acts locally nilpotently on it.
Since $M$ is simple, this criterion implies that $M$ is
finite-dimensional.
Consequently, $M$ is a highest weight module with highest
weight $\mu\in\Lambda^+$.
Let $1\leq i<j\leq n$ and
consider $M$ as a $\mathbb{C}E_{ij}\oplus\mathbb{C}(E_{ii}-E_{jj})\oplus
\mathbb{C}E_{ji}\cong \mathfrak{sl}_2$-module. Then,
Claim 3 implies that the highest weight of $M$ is $0$ or $1$, that is,
$0\leq\mu(E_{ii}-E_{jj})\leq 1$.
Therefore, $M$
is isomorphic to $V(\delta_k)$ for some $k=0,1,\cdots,n$,
which is a contradiction.
\end{proof}

For a Lie algebra or an associative algebra $\mathfrak{G}$ and a $\mathfrak{G}$-module
$V$, we denote by $\mathrm{Ann}_{\mathfrak{G}}(v)$
the annihilator of $v\in V$ in $\mathfrak{G}$.
The following result gives an isomorphism criterion for 
two irreducible modules $F(P,M)$.

\begin{proposition}
	Let $P$, $P^\prime$ be irreducible $D_n$-modules
and $M$, $M^\prime$ be irreducible $\mathfrak{sl}_n$-modules. Suppose that 
$M\ncong V(\delta_r)$ for $r=0,1,\cdots,n$.
Then $F(P,M)\cong F(P^\prime,M^\prime)$ if and only if $P\cong P^\prime$
and $M\cong M^\prime$.
\end{proposition}
\begin{proof}
The sufficiency is obvious. Now suppose that
\[
\psi:F(P,M)\rightarrow F(P^\prime,M^\prime)
\]
is an isomorphism of $S_n$-modules.
By the argument at the end of the proof of Theorem \ref{t1},
the assumption that $M\ncong V(\delta_r)$ for $r=0,1,\cdots,n$
implies that there exist $i\ne j$ and $v\in M$ such that
$(E_{ij})^2v\ne0$.
Fix such $i,j$ and $v$, and choose $0\ne p\in P$.
Write
\[
\psi(p\otimes v)=\sum_{k=1}^q p_k^\prime\otimes v_k^\prime,
\]
where $p_1^\prime,p_2^\prime,\cdots,p_q^\prime$ are linearly independent.

As in Claim 1 in the proof of Theorem \ref{t1}, we have
\begin{equation}\label{eq2}
\psi(xp\otimes (E_{rs})^2v)
=
\sum_{k=1}^q xp_k^\prime\otimes (E_{rs})^2v_k^\prime,
\end{equation}
for all $x\in D_n$ and all $1\leq r,s\leq n$ with $r\ne s$.
Taking $x=1$, $r=i$, and $s=j$, and using the injectivity of $\psi$,
we obtain
\[
0\ne\psi(p\otimes (E_{ij})^2v)
=
\sum_{k=1}^q p_k^\prime\otimes (E_{ij})^2v_k^\prime.
\]
Hence $(E_{ij})^2v_k^\prime\ne0$ for some $k$.
After reindexing the summands, we may assume that
$(E_{ij})^2v_1^\prime\ne0$.

As $p_1^\prime,p_2^\prime,\cdots,p_q^\prime$ are linearly independent,
the density theorem gives $y\in D_n$ such that
\[
yp_k^\prime=\delta_{k1}p_1^\prime
\qquad(k=1,2,\cdots,q).
\]
Applying (\ref{eq2}) with $x=y$, $r=i$, and $s=j$, we get
\[
\psi(yp\otimes (E_{ij})^2v)
=
p_1^\prime\otimes (E_{ij})^2v_1^\prime\ne0.
\]
In particular, $yp\ne0$.
Replacing $x$ by $xy$ in (\ref{eq2}) and taking $r=i$, $s=j$,
we obtain
\[
\psi(xyp\otimes (E_{ij})^2v)
=
xp_1^\prime\otimes (E_{ij})^2v_1^\prime,
\]
for all $x\in D_n$.
Now relabel $yp$ as $p$ and $(E_{ij})^2v$ as $v$, and set
$v^\prime=(E_{ij})^2v_1^\prime$.
Then $p,v,v^\prime\ne0$, and
\begin{equation}\label{eq3}
\psi(xp\otimes v)=xp_1^\prime\otimes v^\prime,
\end{equation}
for all $x\in D_n$.

Since $\psi$ is injective and $v,v^\prime\ne0$, (\ref{eq3}) gives
\[
\mathrm{Ann}_{D_n}(p)
=
\mathrm{Ann}_{D_n}(p_1^\prime).
\]
Moreover, the simplicity of $P$ and $P^\prime$ implies that
\[
P=D_np,
\qquad
P^\prime=D_np_1^\prime.
\]
Thus the map $\psi_1:P\rightarrow P^\prime$ defined by
\[
\psi_1(xp)=xp_1^\prime
\qquad(x\in D_n)
\]
is a well-defined $D_n$-module isomorphism, where $p$ is the
fixed nonzero vector in (\ref{eq3}).
Consequently,
\begin{equation}\label{eq4}
\psi(z\otimes v)=\psi_1(z)\otimes v^\prime,
\end{equation}
for all $z\in P$.

For any $1\leq a,b\leq n$ with $a\ne b$ and any $z\in P$,
the $S_n$-linearity of $\psi$ gives
\[
\psi\bigl((t_a\partial_b)(z\otimes v)\bigr)
=
(t_a\partial_b)\psi(z\otimes v).
\]
Expanding both sides and using (\ref{eq4}), we obtain
\[
\begin{aligned}
&\psi\bigl((t_a\partial_bz)\otimes v\bigr)
+\psi(z\otimes E_{ab}v)\\
&\qquad=
(t_a\partial_b\psi_1(z))\otimes v^\prime
+\psi_1(z)\otimes E_{ab}v^\prime.
\end{aligned}
\]
By (\ref{eq4}) and the $D_n$-linearity of $\psi_1$,
\[
\psi\bigl((t_a\partial_bz)\otimes v\bigr)
=
\psi_1(t_a\partial_bz)\otimes v^\prime
=
(t_a\partial_b\psi_1(z))\otimes v^\prime.
\]
Subtracting these terms yields
\[
\psi(z\otimes E_{ab}v)
=
\psi_1(z)\otimes E_{ab}v^\prime.
\]
Repeating this calculation and using the $D_n$-linearity of
$\psi_1$, we obtain by induction on the length of products of the off-diagonal
matrix units, which generate $U(\mathfrak{sl}_n)$ as a unital
associative algebra, that
\[
\psi(z\otimes uv)=\psi_1(z)\otimes uv^\prime,
\]
for all $u\in U(\mathfrak{sl}_n)$ and all $z\in P$.

Fix $0\ne z\in P$. Since $\psi_1(z)\ne0$ and $\psi$ is injective,
the preceding identity implies that
\[
uv=0
\quad\Longleftrightarrow\quad
uv^\prime=0
\qquad\bigl(u\in U(\mathfrak{sl}_n)\bigr).
\]
Therefore,
\[
\mathrm{Ann}_{U(\mathfrak{sl}_n)}(v)
=
\mathrm{Ann}_{U(\mathfrak{sl}_n)}(v^\prime).
\]
Since $M$ and $M^\prime$ are irreducible $\mathfrak{sl}_n$-modules
and $v,v^\prime\ne0$, we have
\[
M=U(\mathfrak{sl}_n)v,
\qquad
M^\prime=U(\mathfrak{sl}_n)v^\prime.
\]
It follows that
\[
M
\cong
U(\mathfrak{sl}_n)/\mathrm{Ann}_{U(\mathfrak{sl}_n)}(v)
=
U(\mathfrak{sl}_n)/\mathrm{Ann}_{U(\mathfrak{sl}_n)}(v^\prime)
\cong
M^\prime.
\]
This completes the proof.
\end{proof}

\subsection{Subquotients of Shen-Larsson modules}

From Theorem \ref{t1}, the $S_n$-module $F(P,M)$
is simple whenever
$P$ is a simple $D_n$-module and 
$M$ is a simple $\mathfrak{sl}_n$-module
not isomorphic to any of the modules
$V(\delta_k)$, $k=0,1,\cdots,n$.
In this subsection, we focus on 
studying the $S_n$-modules $F(P,V(\delta_r))$, where
$P$ is a simple $D_n$-module and $r=0,1,\cdots,n$.

Consider the following subspace of $F(P,V(\delta_0))=P$:
$$\Delta P:=\sum_{i=1}^n\partial_i P.$$
We claim that
$\Delta P$ is an $S_n$-submodule of $F(P,V(\delta_0))=P$
and the quotient $P/\Delta P$
is trivial.
In fact, for any $p\in P$, 
$1\leq i,j\leq n$ with $i\ne j$ and $\alpha\geq -e_i-e_j$,
we have 
	\begin{align*}
		L_{ij}^{\alpha}p=&\left( 1+\alpha _j \right) t^{\alpha +e_i}\partial _ip-\left( 1+\alpha _i \right) t^{\alpha +e_j}\partial _jp
	\\
	=&\left( 1+\alpha _j \right) \left( \partial _it^{\alpha +e_i}-\left( 1+\alpha _i \right) t^{\alpha} \right) p
	\\
	&-\left( 1+\alpha _i \right) \left( \partial _jt^{\alpha +e_j}-\left( 1+\alpha _j \right) t^{\alpha} \right) p
	\\
	=&\left( 1+\alpha _j \right) \partial _it^{\alpha +e_i}p-\left( 1+\alpha _i \right) \partial _jt^{\alpha +e_j}p
	\\
	\in & \Delta P.
	\end{align*}
This shows that $S_n P\subseteq \Delta P$, as desired.

\begin{proposition}\label{pro1}
	Let $P$ be a simple $D_n$-module.
	The following statements hold.

	(a) If $P\ncong A_n$, then $F(P,V(\delta_0))=P$ has a unique simple $S_n$-submodule
	$\Delta P$ and the quotient $P/\Delta P$
	is trivial.

	(b) $F(A_n,V(\delta_0))=A_n$ has a unique nonzero proper $S_n$-submodule $\mathbb{C}t^0$
	and therefore has a unique
	simple quotient $A_n/\mathbb{C}t^0$.
\end{proposition}
\begin{proof}
	Let $N$ be a nonzero submodule of $F(P,V(\delta_0))=P$.
	
	\textbf{Claim 1.}
	We have $\partial_jD_n\partial_jN\subseteq N$ for any $j=1,2,\cdots,n$.

	Take any $p\in N$. For any $i,j=1,2,\cdots,n$ with $i\ne j$, 
	$\alpha\in \mathbb{Z}_+^n$ and $l=0 ,1$,
	we have 
	\begin{align}
	L_{ij}^{\alpha -le_i}\cdot t^{le_i}\partial _jp=& t^{\alpha -le_i}\left( \left( 1+\alpha _j \right) \left( t^{le_i}d_i+lt^{le_i} \right) -\left( 1+\alpha _i-l \right) t^{le_i}d_j \right) \partial _jp
\notag\\
=&t^{\alpha}\left( \left( 1+\alpha _j \right) \left( d_i+l \right) -\left( 1+\alpha _i-l \right) d_j \right) \partial _jp
\label{eq5}\\
=&\left( 1+\alpha _j \right) t^{\alpha}\left( d_i+l \right) \partial _jp-t^{\alpha}\left( 1+\alpha _i-l \right) d_j\partial _jp\in N.
	\notag\end{align}
	Subtracting the case $l=0$ from the case $l=1$ in (\ref{eq5}), we obtain
$$
		\left( 1+\alpha _j \right) t^{\alpha}\partial _jp+t^{\alpha}d_j\partial _jp=\partial _jt^{\alpha +e_j}\partial _jp\in N,
	$$
	which shows that 
	\begin{equation}\label{eq12}
		\partial_jt^{\alpha+e_j}\partial_jN\subseteq N.
	\end{equation}
	By applying the action of $\partial_j$ to  
	$\partial_jt^{\alpha+e_j}\partial_jp\in N$, we have
	\begin{equation}\label{eq13}
			\partial _j\cdot(\partial _jt^{\alpha +e_j}\partial _jp)
	=\partial _j\partial _jt^{\alpha +e_j}\partial _jp
	=\partial _jt^{\alpha +e_j}\partial _j\partial _jp
	+
	\left( 1+\alpha _j \right) \partial _jt^{\alpha }\partial _jp\in N.
		\end{equation}
From (\ref{eq12}), we can see that
$\partial _jt^{\alpha +e_j}\partial _j\partial _jp\in N$.
Now (\ref{eq13}) implies that
\begin{equation}\label{eq122}
	\partial _jt^{\alpha }\partial _jp\in N.
\end{equation}
Replacing $p$ with $\partial^\beta p\in N$ in (\ref{eq122})
for any $\beta\in \mathbb{Z}_+^n$,
we have $\partial _jt^{\alpha}\partial^\beta\partial _jp\in N$.
Since $D_n$ is generated by $t_r,
\partial_s$ for all $1\leq r,s\leq n$, we obtain that
$$
	\partial _jD_n\partial _jp\subseteq N.
$$
Claim 1 follows.

Now let $p\in N$ be a nonzero element.
We divide the following discussion into two cases.

\textbf{Case i.} There exists some $i_0$ such that $\partial_{i_0}p\ne0$.

For any $1\leq j\leq n$ with $\partial_jp\ne 0$,
since $P$ is a simple $D_n$-module, by Claim 1, 
we have $$\partial_{j}D_n\partial_{j}p=\partial_{j}P\subseteq N.$$
For any $1\leq j\leq n$ with $\partial_jp= 0$, we note that 
$$\partial_jt_j\partial_{i_0}p=(t_j\partial_j+1)\partial_{i_0}p
=t_j\partial_{i_0}\partial_jp+\partial_{i_0}p
=\partial_{i_0}p\ne 0.$$
Then, by Claim 1, we have
$$\partial_{j}D_n\partial_jt_j\partial_{i_0}p=\partial_{j}P\subseteq N.$$
Now we can see that $\Delta P\subseteq N$.

\textbf{Case ii.} $\partial_j p = 0$ for all $j = 1, 2, \cdots,n$. 

In this case, as a $D_n$-module, $P$ is a quotient of $D_n/I = A_n$, where $I$ is the 
left ideal of $D_n$ generated by $\partial_1, \partial_2, \cdots, \partial_n$. 
By Lemma \ref{l2}, $A_n$ is a simple $D_n$-module and therefore
$P$ is isomorphic to $A_n$ as a $D_n$-module. It is easy 
to see that $\mathbb{C}t^0$ is an $S_n$-submodule of $A_n$. Now if $N$ is a 
nonzero $S_n$-submodule of $A_n$ other than $\mathbb{C}t^0$, 
there must exist some $i_0$ and some $p^\prime\in N$ 
such that $\partial_{i_0}p^\prime\ne 0$.
By Case i, we have $\Delta A_n = A_n \subseteq N$, forcing $N = A_n$. 
Hence, $\mathbb{C}t^0$ is the unique nonzero proper $S_n$-submodule of $A_n$.

Recall that we have proved that $\Delta P$ is an $S_n$-submodule of $P$ and $P/\Delta P$
is trivial.
Now (a) follows from Case i and (b) follows from Case ii.
\end{proof}

The proof of Proposition \ref{pro1} also gives the following observation:
if $P$ is a simple $D_n$-module and there exists some nonzero $p\in P$ such that
$\partial_ip=0$ for all $i=1,2,\cdots,n$, then $P\cong A_n$.
We will use this conclusion without further explanation later.

\begin{proposition}\label{pro57}
	Let $P$ be a simple $D_n$-module. The following statements hold.

	(a) If $P\ncong A_n$, then $L_n(P,1)\cong F(P,V(\delta_0))$ as $S_n$-modules.

	(b) If $P\cong A_n$, then $L_n(P,1)\cong A_n/\mathbb{C}t^0$ is a simple $S_n$-module.
\end{proposition}
\begin{proof}
	We note that 
	$$\mathrm{Ker}(\pi_0)=\{p\in F(P,V(\delta_0))=P|\partial_ip=0,\forall i=1,2,\cdots,n\},$$
	which is nonzero if and only if  
	$P\cong A_n$.
	If $P\ncong A_n$, $\pi_0$ is injective, 
	and hence $L_n(P,1)=\mathrm{Im}(\pi_0)\cong F(P,V(\delta_0))$.
	If $P\cong A_n$, then $\mathrm{Ker}(\pi_0)=\mathbb{C}t^0$ and 
	$L_n(A_n,1)=\mathrm{Im}(\pi_0)\cong A_n/\mathbb{C}t^0$ 
	is simple by Proposition \ref{pro1}(b).
\end{proof}

Now we turn to studying $S_n$-modules $F(P,V(\delta_r))$ with
$2\leq r\leq n-1$.
We need some calculations here.
As before, let $\hat{\iota}$ be the Lie algebra homomorphism from $\mathcal{W}_n$
	to $\mathcal{D}_n\otimes U(\mathfrak{gl}_n)$
	defined by extending the domain of $\alpha$ to $\mathbb{Z}^n$ in equation (\ref{eq16}).
Let $\alpha\in\mathbb{Z}^n$, $m\in\mathbb{Z}$ and $1\leq i\leq n-2$. We have
\begin{align*}
	&\hat{\iota} \left( L_{i,i+2}^{\alpha -me_i} \right) \cdot 
	\hat{\iota} \left( L_{i,i+1}^{me_i} \right) 
	\\
	=&\left( 1+\alpha _{i+2} \right) \hat{\iota} \left( t^{\alpha -me_i}d_i \right) 
	\cdot\hat{\iota} \left( t^{me_i}d_i \right) 
	-\left( 1+m \right) \left( 1+\alpha _{i+2} \right) 
	\hat{\iota} \left( t^{\alpha -me_i}d_i \right) \cdot\hat{\iota} \left( t^{me_i}d_{i+1} \right) 
	\\
	&-\left( 1+\alpha _i-m \right) \hat{\iota} \left( t^{\alpha -me_i}d_{i+2} \right) 
	\cdot\hat{\iota} \left( t^{me_i}d_i \right) 
	+\left( 1+m \right) \left( 1+\alpha _i-m \right) 
	\hat{\iota} \left( t^{\alpha -me_i}d_{i+2} \right) \cdot\hat{\iota} 
	\left( t^{me_i}d_{i+1} \right) 
	\\
	=&\left( 1+\alpha _{i+2} \right) \left( t^{\alpha -\left( m-1 \right) e_i}\partial _i\otimes 1+\sum_{s=1}^n{\left( \alpha _s-\delta _{si}\left( m-1 \right) \right) t^{\alpha -\left( m-1 \right) e_i-e_s}\otimes E_{si}} \right) 
	\\
	&\phantom{\left( 1+\alpha _{i+2} \right) }\cdot \left( t^{\left( m+1 \right) e_i}\partial _i\otimes 1+\left( m+1 \right) t^{me_i}\otimes E_{ii} \right) 
	\\
	&-\left( 1+m \right) \left( 1+\alpha _{i+2} \right) \left( t^{\alpha -\left( m-1 \right) e_i}\partial _i\otimes 1+\sum_{s=1}^n{\left( \alpha _s-\delta _{si}\left( m-1 \right) \right) t^{\alpha -\left( m-1 \right) e_i-e_s}\otimes E_{si}} \right) 
	\\
	&\phantom{-\left( 1+m \right) \left( 1+\alpha _{i+2} \right) }\cdot \left( t^{me_i+e_{i+1}}\partial _{i+1}\otimes 1+mt^{\left( m-1 \right) e_i+e_{i+1}}\otimes E_{i,i+1}+t^{me_i}\otimes E_{i+1,i+1} \right) 
	\\
	&-\left( 1+\alpha _i-m \right) \left( t^{\alpha -me_i+e_{i+2}}\partial _{i+2}\otimes 1+\sum_{s=1}^n{\left( \alpha _s-\delta _{si}m+\delta _{s,i+2} \right) t^{\alpha -me_i+e_{i+2}-e_s}\otimes E_{s,i+2}} \right) 
	\\
	&\phantom{-\left( 1+\alpha _i-m \right) }\cdot \left( t^{\left( m+1 \right) e_i}\partial _i\otimes 1+\left( m+1 \right) t^{me_i}\otimes E_{ii} \right) 
	\\
	&+\left( 1+m \right) \left( 1+\alpha _i-m \right) \left( t^{\alpha -me_i+e_{i+2}}\partial _{i+2}\otimes 1+\sum_{s=1}^n{\left( \alpha _s-\delta _{si}m+\delta _{s,i+2} \right) t^{\alpha -me_i+e_{i+2}-e_s}\otimes E_{s,i+2}} \right) 
	\\
	&\phantom{+\left( 1+m \right) \left( 1+\alpha _i-m \right) }\cdot \left( t^{me_i+e_{i+1}}\partial _{i+1}\otimes 1+mt^{\left( m-1 \right) e_i+e_{i+1}}\otimes E_{i,i+1}+t^{me_i}\otimes E_{i+1,i+1} \right) 
	\\
	=&\left( 1+\alpha _{i+2} \right) t^{\alpha +2e_i}\partial _i\partial _i\otimes 1
	+\left( m+1 \right) \left( 1+\alpha _{i+2} \right) t^{\alpha +e_i}\partial _i\otimes 1
	\\
	&+\left( 1+\alpha _{i+2} \right) \sum_{s=1}^n{\left( \alpha _s-\delta _{si}\left( m-1 \right) \right) t^{\alpha +2e_i-e_s}\partial _i\otimes E_{si}}
	\\
	&+\left( m+1 \right) \left( 1+\alpha _{i+2} \right) t^{\alpha +e_i}\partial _i\otimes 1
	+m\left( m+1 \right) \left( 1+\alpha _{i+2} \right) t^{\alpha}\otimes 1
	\\
	&+\left( m+1 \right) \left( 1+\alpha _{i+2} \right) \sum_{s=1}^n{\left( \alpha _s-\delta _{si}\left( m-1 \right) \right) t^{\alpha +e_i-e_s}\otimes E_{si}}
	\\
	&-\left( 1+m \right) \left( 1+\alpha _{i+2} \right) t^{\alpha +e_i+e_{i+1}}\partial _i\partial _{i+1}\otimes 1
	-m\left( 1+m \right) \left( 1+\alpha _{i+2} \right) t^{\alpha +e_{i+1}}\partial _{i+1}\otimes 1
	\\
	&-\left( 1+m \right) \left( 1+\alpha _{i+2} \right) \sum_{s=1}^n{\left( \alpha _s-\delta _{si}\left( m-1 \right) \right) t^{\alpha +e_i+e_{i+1}-e_s}\partial _{i+1}\otimes E_{si}}
	\\
	&-m\left( 1+m \right) \left( 1+\alpha _{i+2} \right) t^{\alpha +e_{i+1}}\partial _i\otimes E_{i,i+1}
	-m\left( m-1 \right) \left( 1+m \right) \left( 1+\alpha _{i+2} \right) t^{\alpha -e_i+e_{i+1}}\otimes E_{i,i+1}
	\\
	&-m\left( 1+m \right) \left( 1+\alpha _{i+2} \right) \sum_{s=1}^n{\left( \alpha _s-\delta _{si}\left( m-1 \right) \right) t^{\alpha +e_{i+1}-e_s}\otimes E_{si}E_{i,i+1}}
	\\
	&-\left( 1+m \right) \left( 1+\alpha _{i+2} \right) t^{\alpha +e_i}\partial _i\otimes E_{i+1,i+1}
	-m\left( 1+m \right) \left( 1+\alpha _{i+2} \right) t^{\alpha}\otimes E_{i+1,i+1}
	\\
	&-\left( 1+m \right) \left( 1+\alpha _{i+2} \right) \sum_{s=1}^n{\left( \alpha _s-\delta _{si}\left( m-1 \right) \right) t^{\alpha +e_i-e_s}\otimes E_{si}E_{i+1,i+1}}
	\\
	&-\left( 1+\alpha _i-m \right) t^{\alpha +e_i+e_{i+2}}\partial _{i+2}\partial _i\otimes 1
	\\
	&-\left( 1+\alpha _i-m \right) \sum_{s=1}^n{\left( \alpha _s-\delta _{si}m+\delta _{s,i+2} \right) t^{\alpha +e_i+e_{i+2}-e_s}\partial _i\otimes E_{s,i+2}}
	\\
	&-\left( 1+\alpha _i-m \right) \left( m+1 \right) t^{\alpha +e_{i+2}}\partial _{i+2}\otimes E_{ii}
	\\
	&-\left( 1+\alpha _i-m \right) \left( m+1 \right) \sum_{s=1}^n{\left( \alpha _s-\delta _{si}m+\delta _{s,i+2} \right) t^{\alpha +e_{i+2}-e_s}\otimes E_{s,i+2}E_{ii}}
	\\
	&+\left( 1+m \right) \left( 1+\alpha _i-m \right) t^{\alpha +e_{i+1}+e_{i+2}}\partial _{i+1}\partial _{i+2}\otimes 1
	\\
	&+\left( 1+m \right) \left( 1+\alpha _i-m \right) \sum_{s=1}^n{\left( \alpha _s-\delta _{si}m+\delta _{s,i+2} \right) t^{\alpha +e_{i+1}+e_{i+2}-e_s}\partial _{i+1}\otimes E_{s,i+2}}
	\\
	&+\left( 1+m \right) \left( 1+\alpha _i-m \right) mt^{\alpha -e_i+e_{i+1}+e_{i+2}}\partial _{i+2}\otimes E_{i,i+1}
	\\
	&+\left( 1+m \right) \left( 1+\alpha _i-m \right) m\sum_{s=1}^n{\left( \alpha _s-\delta _{si}m+\delta _{s,i+2} \right) t^{\alpha -e_i+e_{i+1}+e_{i+2}-e_s}\otimes E_{s,i+2}E_{i,i+1}}
	\\
	&+\left( 1+m \right) \left( 1+\alpha _i-m \right) t^{\alpha +e_{i+2}}\partial _{i+2}\otimes E_{i+1,i+1}
	\\
	&+\left( 1+m \right) \left( 1+\alpha _i-m \right) \sum_{s=1}^n{\left( \alpha _s-\delta _{si}m+\delta _{s,i+2} \right) t^{\alpha +e_{i+2}-e_s}\otimes E_{s,i+2}E_{i+1,i+1}}
	.
	\end{align*}
Then we can write
\begin{equation}\label{eq9}
	\hat{\iota} \left( L_{i,i+2}^{\alpha -me_i} \right) \cdot \hat{\iota} \left( L_{i,i+1}^{me_i} \right)
	=m^4z_4+m^3g(\alpha,i)+m^2z_2+mz_1+z_0
\end{equation}
where $z_4,z_2,z_1,z_0\in\mathcal{D}_n\otimes U(\mathfrak{gl}_n)$ are independent
of $m$ and 
\begin{align*}
	g(\alpha,i):=&\left( 1+\alpha _{i+2} \right) t^{\alpha -e_i+e_{i+1}}\otimes \left( E_{ii}E_{i,i+1}-E_{i,i+1} \right) 
 \\
 &-t^{\alpha -e_i+e_{i+2}}\otimes E_{i,i+2}E_{ii}
 +t^{\alpha +e_{i+1}+e_{i+2}-e_i}\partial _{i+1}\otimes E_{i,i+2}
 \\
 &-t^{\alpha +e_{i+1}+e_{i+2}-e_i}\partial _{i+2}\otimes E_{i,i+1}
 -\sum_{s=1}^n{\alpha _st^{\alpha +e_{i+1}+e_{i+2}-e_i-e_s}\otimes E_{s,i+2}E_{i,i+1}}
 \\
 &-t^{\alpha +e_{i+1}-e_i}\otimes E_{i+2,i+2}E_{i,i+1}
 -\alpha _it^{\alpha +e_{i+1}+e_{i+2}-2e_i}\otimes E_{i,i+2}E_{i,i+1}
\\&+t^{\alpha +e_{i+2}-e_i}\otimes E_{i,i+2}E_{i+1,i+1}
.
 \end{align*}
Setting $m=-1,0,1,2,3$ in (\ref{eq9}), we get a linear system of equations
whose coefficient matrix is nonsingular.
 Then we obtain that
 \begin{equation}\label{eq10}
	\begin{aligned}
		g\left( \alpha ,i \right) =&-\frac{1}{12}\hat{\iota} \left( L_{i,i+2}^{\alpha -3e_i} 
		\right) \cdot \hat{\iota} \left( L_{i,i+1}^{3e_i} \right) 
		+\frac{1}{2}\hat{\iota} \left( L_{i,i+2}^{\alpha -2e_i} \right) \cdot 
		\hat{\iota} \left( L_{i,i+1}^{2e_i} \right) 
\\
&-\hat{\iota} \left( L_{i,i+2}^{\alpha -e_i} \right) \cdot 
\hat{\iota} \left( L_{i,i+1}^{e_i} \right) +\frac{5}{6}\hat{\iota} 
\left( L_{i,i+2}^{\alpha} \right) \cdot \hat{\iota} \left( L_{i,i+1}^{0} \right) 
\\
&-\frac{1}{4}\hat{\iota} \left( L_{i,i+2}^{\alpha +e_i} \right) 
\cdot \hat{\iota} \left( L_{i,i+1}^{-e_i} \right) 
.
	\end{aligned}
 \end{equation}
Note that if $\alpha\geq 2e_i-e_{i+2}$, the vector fields occurring
in (\ref{eq10}) belong to $S_n$.
Their products therefore belong to $U(S_n)$.
Since $\hat{\iota}$ agrees with $\iota$ on $S_n$, each product
on the right-hand side of (\ref{eq10}) lies in $\iota(U(S_n))$.
Hence, $g(\alpha,i)\in\iota(U(S_n))$.

For convenience, we set
 \begin{align*}
	 f(\alpha,i):=&\left( 1+\alpha _{i+2} \right) t^{\alpha -e_i+e_{i+1}}\otimes \left( E_{ii}E_{i,i+1}-E_{i,i+1} \right) 
	 \\
	 &-t^{\alpha -e_i+e_{i+2}}\otimes E_{i,i+2}E_{ii}
	 \\
	 &-\alpha _it^{\alpha +e_{i+1}+e_{i+2}-2e_i}\otimes E_{i,i+2}E_{i,i+1}.
 \end{align*}
 Then, we have
 \begin{align*}
	 g(\alpha,i)-f(\alpha,i)=&t^{\alpha +e_{i+1}+e_{i+2}-e_i}\partial _{i+1}\otimes E_{i,i+2}-t^{\alpha +e_{i+1}+e_{i+2}-e_i}\partial _{i+2}\otimes E_{i,i+1}
 \\
 &-\sum_{s=1}^n{\alpha _st^{\alpha +e_{i+1}+e_{i+2}-e_i-e_s}\otimes E_{s,i+2}E_{i,i+1}}
 \\
 &-t^{\alpha +e_{i+1}-e_i}\otimes E_{i+2,i+2}E_{i,i+1}+t^{\alpha +e_{i+2}-e_i}\otimes E_{i,i+2}E_{i+1,i+1}
 \\
 =&t^{\alpha +e_{i+1}+e_{i+2}-e_i}\partial _{i+1}\otimes E_{i,i+2}-t^{\alpha +e_{i+1}+e_{i+2}-e_i}\partial _{i+2}\otimes E_{i,i+1}
 \\
 &-\sum_{s=1}^n{\alpha _st^{\alpha +e_{i+1}+e_{i+2}-e_i-e_s}\otimes E_{s,i+2}E_{i,i+1}}-t^{\alpha +e_{i+1}-e_i}\otimes E_{i+2,i+2}E_{i,i+1}
 \\
 &+t^{\alpha +e_{i+2}-e_i}\otimes E_{i,i+2}E_{i+1,i+1}+\sum_{s=1}^n{\partial _s\left( t^{\alpha +e_{i+1}+e_{i+2}-e_i} \right) \otimes E_{s,i+2}E_{i,i+1}}
 \\
 &-\sum_{s=1}^n{\partial _s\left( t^{\alpha +e_{i+1}+e_{i+2}-e_i} \right) \otimes E_{s,i+2}E_{i,i+1}}
 \\
 =&t^{\alpha +e_{i+1}+e_{i+2}-e_i}\partial _{i+1}\otimes E_{i,i+2}-t^{\alpha +e_{i+1}+e_{i+2}-e_i}\partial _{i+2}\otimes E_{i,i+1}
 \\
 &-\sum_{s=1}^n{\alpha _st^{\alpha +e_{i+1}+e_{i+2}-e_i-e_s}\otimes E_{s,i+2}E_{i,i+1}}
 \\
 &-t^{\alpha +e_{i+1}-e_i}\otimes E_{i+2,i+2}E_{i,i+1}+t^{\alpha +e_{i+2}-e_i}\otimes E_{i,i+2}E_{i+1,i+1}
 \\
 &+\sum_{s=1}^n{\left( \alpha _s+\delta _{s,i+1}+\delta _{s,i+2}-\delta _{si} \right) t^{\alpha +e_{i+1}+e_{i+2}-e_i-e_s}\otimes E_{s,i+2}E_{i,i+1}}
 \\
 &-\sum_{s=1}^n{\left( \partial _st^{\alpha +e_{i+1}+e_{i+2}-e_i}-t^{\alpha +e_{i+1}+e_{i+2}-e_i}\partial _s \right) \otimes E_{s,i+2}E_{i,i+1}}
 \\
 =&t^{\alpha +e_{i+1}+e_{i+2}-e_i}\partial _{i+1}\otimes E_{i,i+2}-t^{\alpha +e_{i+1}+e_{i+2}-e_i}\partial _{i+2}\otimes E_{i,i+1}
 \\
 &-\sum_{s=1}^n{\alpha _st^{\alpha +e_{i+1}+e_{i+2}-e_i-e_s}\otimes E_{s,i+2}E_{i,i+1}}
 \\
 &-t^{\alpha +e_{i+1}-e_i}\otimes E_{i+2,i+2}E_{i,i+1}+t^{\alpha +e_{i+2}-e_i}\otimes E_{i,i+2}E_{i+1,i+1}
 \\
 &+\sum_{s=1}^n{\alpha _st^{\alpha +e_{i+1}+e_{i+2}-e_i-e_s}\otimes E_{s,i+2}E_{i,i+1}}
 \\
 &+t^{\alpha +e_{i+2}-e_i}\otimes E_{i+1,i+2}E_{i,i+1}+t^{\alpha +e_{i+1}-e_i}\otimes E_{i+2,i+2}E_{i,i+1}
 \\
 &-t^{\alpha +e_{i+1}+e_{i+2}-2e_i}\otimes E_{i,i+2}E_{i,i+1}
 \\
 &-\sum_{s=1}^n{\partial _st^{\alpha +e_{i+1}+e_{i+2}-e_i}\otimes E_{s,i+2}E_{i,i+1}}
 \\
 &+\sum_{s=1}^n{t^{\alpha +e_{i+1}+e_{i+2}-e_i}\partial _s\otimes E_{s,i+2}E_{i,i+1}}
 \\
 =&u(\alpha,i)
 +t^{\alpha +e_{i+2}-e_i}\otimes (E_{i,i+2}E_{i+1,i+1}+E_{i+1,i+2}E_{i,i+1})
 \\
 &-t^{\alpha +e_{i+1}+e_{i+2}-2e_i}\otimes E_{i,i+2}E_{i,i+1},
 \end{align*}
where
 \begin{align*}
	 u(\alpha,i):=&t^{\alpha +e_{i+1}+e_{i+2}-e_i}\partial _{i+1}\otimes E_{i,i+2}-t^{\alpha +e_{i+1}+e_{i+2}-e_i}\partial _{i+2}\otimes E_{i,i+1}
 \\
 &-\sum_{s=1}^n{\partial _st^{\alpha +e_{i+1}+e_{i+2}-e_i}\otimes E_{s,i+2}E_{i,i+1}}+\sum_{s=1}^n{t^{\alpha +e_{i+1}+e_{i+2}-e_i}\partial _s\otimes E_{s,i+2}E_{i,i+1}}
 .
 \end{align*}
 
 \begin{lemma}\label{l3}
 Let $n\geq 3$, $2\leq r\leq n-1$, $\alpha\geq 2e_i-e_{i+2}$ and $P$ be a simple $D_n$-module. 
For any $p\otimes v\in F(P,V(\delta_r))$, 
 we have $g(\alpha,i)(p\otimes v)=u(\alpha,i)(p\otimes v)$.
 \end{lemma}
 \begin{proof}
	 It is sufficient to prove the statement for all 
	 $v=\varepsilon _{i_1}\land \varepsilon _{i_2}\land \cdots \land \varepsilon _{i_r}$,
	where $i_1,i_2,\cdots,i_r$ are pairwise distinct.
	Note that 
	\begin{align*}
		g(\alpha,i)-u(\alpha,i)
		=&\left( 1+\alpha _{i+2} \right) t^{\alpha -e_i+e_{i+1}}
		\otimes \left( E_{ii}E_{i,i+1}-E_{i,i+1} \right) 
		-t^{\alpha -e_i+e_{i+2}}\otimes E_{i,i+2}E_{ii}
		\\
		&-\alpha _it^{\alpha +e_{i+1}+e_{i+2}-2e_i}\otimes E_{i,i+2}E_{i,i+1}
		\\
		&+t^{\alpha +e_{i+2}-e_i}\otimes (E_{i,i+2}E_{i+1,i+1}+E_{i+1,i+2}E_{i,i+1})
 \\
 &-t^{\alpha +e_{i+1}+e_{i+2}-2e_i}\otimes E_{i,i+2}E_{i,i+1}.
	\end{align*}

	Firstly, it is easy to see that
	 $$\left( E_{ii}E_{i,i+1}-E_{i,i+1} \right) v=E_{i,i+2}E_{ii}v=E_{i,i+2}E_{i,i+1}v
 =0.$$
 
 Secondly, we have $$(E_{i,i+2}E_{i+1,i+1}+E_{i+1,i+2}E_{i,i+1})v=0$$
 unless $i\notin \left\{ i_1,i_2,\cdots ,i_r \right\}$ and 
 $i+1,i+2\in\left\{ i_1,i_2,\cdots ,i_r \right\}$.
 Without loss of generality, we can assume that
 $v=\varepsilon _{i+1}\land \varepsilon _{i+2}\land \varepsilon _{k_3}\land \cdots \land \varepsilon _{k_r}$,
 where $k_3,\ldots,k_r$ are the remaining indices; the trailing exterior
 product is omitted when $r=2$.
 Then 
 \begin{align*}
	 &(E_{i,i+2}E_{i+1,i+1}+E_{i+1,i+2}E_{i,i+1})v
 \\*
 =&(E_{i,i+2}E_{i+1,i+1}+E_{i+1,i+2}E_{i,i+1})\left( \varepsilon _{i+1}\land \varepsilon _{i+2}\land \varepsilon _{k_3}\land \cdots \land \varepsilon _{k_r} \right) 
 \\*
 =&\varepsilon _{i+1}\land \varepsilon _i\land \varepsilon _{k_3}\land \cdots \land \varepsilon _{k_r}+\varepsilon _i\land \varepsilon _{i+1}\land \varepsilon _{k_3}\land \cdots \land \varepsilon _{k_r}
 \\*
 =&0.
 \end{align*}

Thus $(g(\alpha,i)-u(\alpha,i))(p\otimes v)=0$. The lemma follows.
 \end{proof}
 Let
 \begin{align*}
	 h(\alpha,i):=&t^{\alpha +e_{i+1}+e_{i+2}-e_i}\partial _{i+1}\otimes E_{i,i+2}-t^{\alpha +e_{i+1}+e_{i+2}-e_i}\partial _{i+2}\otimes E_{i,i+1}
	 \\
	 &+\sum_{s=1}^n{t^{\alpha +e_{i+1}+e_{i+2}-e_i}\partial _s\otimes E_{s,i+2}E_{i,i+1}}.
 \end{align*}
 Then $u(\alpha,i)=h(\alpha,i)
 -\sum_{s=1}^n{\partial _st^{\alpha +e_{i+1}+e_{i+2}-e_i}\otimes E_{s,i+2}E_{i,i+1}}$.

The proof of the following lemma is similar to that of \cite[Lemma 4.14]{DGYZ}.
 \begin{lemma}\label{l4}
	 Let $n\geq 3$, $2\leq r\leq n-1$, $\alpha\geq 2e_i-e_{i+2}$ 
	 and $P$ be a simple $D_n$-module. 
We have $h(\alpha,i)L_n(P,r)=0$.
 \end{lemma}
\begin{proof}
	Take any $\sum_{l=1}^{n}\partial_lp\otimes E_{lj}w\in L_n(P,r)$
	with $w=\varepsilon_{j_1}\land \varepsilon_{j_2}\land\cdots\land\varepsilon_{j_r}$
	for some pairwise distinct indices $j_1=j,j_2,\cdots,j_r$ in $\{1,\ldots,n\}$.
	We have 
	\begin{equation}\label{eq15}
		\begin{aligned}
h(\alpha ,i)\left( \sum_{l=1}^n{\partial _lp\otimes E_{lj}w} \right) =&\sum_{l=1}^n{t^{\alpha +e_{i+1}+e_{i+2}-e_i}\partial _{i+1}\partial _lp\otimes E_{i,i+2}E_{lj}w}
\\
&-\sum_{l=1}^n{t^{\alpha +e_{i+1}+e_{i+2}-e_i}\partial _{i+2}\partial _lp\otimes E_{i,i+1}E_{lj}w}
\\
&+\sum_{l,s=1}^n{t^{\alpha +e_{i+1}+e_{i+2}-e_i}\partial _s\partial _lp\otimes E_{s,i+2}E_{i,i+1}E_{lj}w}
.
		\end{aligned}
	\end{equation}

	Let $\beta=\alpha+e_{i+1}+e_{i+2}-e_i$.
	The term involving $t^{\beta}\partial_{i+1}^2p$ in (\ref{eq15}) is
	$$t^{\beta}\partial_{i+1}^2p\otimes(E_{i,i+2}E_{i+1,j}+E_{i+1,i+2}E_{i,i+1}E_{i+1,j})w=0.$$
	The term involving $t^{\beta}\partial_{i+2}^2p$ in (\ref{eq15}) is
	$$t^{\beta}\partial_{i+2}^2p\otimes(E_{i+2,i+2}E_{i,i+1}E_{i+2,j}-E_{i,i+1}E_{i+2,j})w=0.$$
	The term involving $t^{\beta}\partial_{i+1}\partial_{i+2}p$ in (\ref{eq15}) is
	$$t^{\beta}\partial_{i+1}\partial_{i+2}p\otimes(E_{i,i+2}E_{i+2,j}-E_{i,i+1}E_{i+1,j}+E_{i+1,i+2}E_{i,i+1}E_{i+2,j}+E_{i+2,i+2}E_{i,i+1}E_{i+1,j})w=0.$$
	The term involving $t^{\beta}\partial_l\partial_{i+1}p$ in (\ref{eq15}) for $l\neq i+1,i+2$ is
	$$t^{\beta}\partial_l\partial_{i+1}p\otimes(E_{i,i+2}E_{lj}+E_{i+1,i+2}E_{i,i+1}E_{l,j}+E_{l,i+2}E_{i,i+1}E_{i+1,j})w=0.$$
	The term involving $t^{\beta}\partial_l\partial_{i+2}p$ in (\ref{eq15}) for $l\neq i+1,i+2$ is
	$$t^{\beta}\partial_l\partial_{i+2}p\otimes(-E_{i,i+1}E_{lj}+E_{i+2,i+2}E_{i,i+1}E_{l,j}+E_{l,i+2}E_{i,i+1}E_{i+2,j})w=0.$$
	The term involving $t^{\beta}\partial_l^2p$ in (\ref{eq15}) for $l\neq i+1,i+2$ is
	$$t^{\beta}\partial_l^2p\otimes E_{l,i+2}E_{i,i+1}E_{l,j}w=0.$$
	The term involving $t^{\beta}\partial_l\partial_{s}p$ in (\ref{eq15}) 
	for $l\neq i+1,i+2$ and $s\neq i+1,i+2$ is
	$$t^{\beta}\partial_l\partial_{s}p\otimes(E_{s,i+2}
	E_{i,i+1}E_{lj}+E_{l,i+2}E_{i,i+1}E_{s,j})w=0.$$

	Hence the right-hand side of (\ref{eq15}) is zero, as desired.
\end{proof}

\begin{lemma}\label{l6}
	Let $n\geq 3$, $2\leq r\leq n-1$, $\alpha\geq 2e_i-e_{i+2}$ 
	and $P$ be a simple $D_n$-module. 
	If $N$ is an $S_n$-submodule of $L_n(P,r)$, we have 
	$$\left(\sum_{s=1}^n{\partial _st^{\alpha +e_{i+1}+e_{i+2}-e_i}\otimes 
	E_{s,i+2}E_{i,i+1}}\right)N\subseteq N.$$
\end{lemma}
\begin{proof}
	Note that $$
 \sum_{s=1}^n{\partial _st^{\alpha +e_{i+1}+e_{i+2}-e_i}\otimes E_{s,i+2}E_{i,i+1}}
 =h(\alpha,i)-u(\alpha,i)$$
and $g(\alpha,i)\in\iota(U(S_n))$.
	Then from Lemma \ref{l3} and Lemma \ref{l4}, we have 
	$$\left(\sum_{s=1}^n{\partial _st^{\alpha +e_{i+1}+e_{i+2}-e_i}\otimes 
	E_{s,i+2}E_{i,i+1}}\right)y=-g(\alpha,i)y\in N$$
	for any $y\in N$.
\end{proof}

Now we give the following result.
\begin{proposition}\label{pro58}
	Let $n\geq 3$, $2\leq r\leq n-1$ and $P$ be a simple $D_n$-module.
	The following statements hold.
	
	(a) $L_n(P,r)$
	is a simple $S_n$-submodule of $F(P,V(\delta_r))$.

	(b) If $r\ne n-1$, $F(P,V(\delta_r))/\widetilde{L_n}(P,r)
	\cong L_n(P,r+1)$ is a simple $S_n$-module.

	(c) $F(A_n,V(\delta_{n-1}))/\widetilde{L_n}(A_n,n-1)\cong A_n$ 
	has a unique simple $S_n$-quotient $A_n/\mathbb{C}t^0$.

	(d) If $P\ncong A_n$, $F(P,V(\delta_{n-1}))/\widetilde{L_n}(P,{n-1})\cong 
	\Delta P$ is a simple $S_n$-module.

\end{proposition}
\begin{proof}
Suppose that $N$ is a nonzero $S_n$-submodule of $L_n(P,r)$.
   Fix a nonzero $y=\sum_{j\in J}p_j\otimes v_j\in N$, where
   $J$ is a finite index set, all $v_j\in V(\delta_r)$, $j\in J$, are nonzero and
   $p_j\in P$, $j\in J$, are linearly independent.

\textbf{Claim 1.}
We can choose $y$ and a nonzero weight component $v$ of one
of the $v_j$ such that
\[
v\in\mathbb{C}
\bigl(\varepsilon_{n-r+1}\wedge\cdots\wedge\varepsilon_n\bigr).
\]

Let $\alpha_1,\ldots,\alpha_{n-1}$ be the standard simple roots
of $\mathfrak{sl}_n$, with $\alpha_a$ corresponding to
$E_{a,a+1}$.
We use the partial order on $\mathfrak{h}^*$ defined by
\[
\mu\preceq\nu
\quad\Longleftrightarrow\quad
\nu-\mu\in\sum_{a=1}^{n-1}\mathbb{Z}_{\geq0}\alpha_a.
\]
We write $\mu\prec\nu$ if $\mu\preceq\nu$ and $\mu\ne\nu$.

For the above nonzero vector
$y=\sum_{j\in J}p_j\otimes v_j\in N$, write
\[
v_j=\sum_{\lambda}v_{j,\lambda},
\qquad
v_{j,\lambda}\in V(\delta_r)_{\lambda}.
\]
Let
\[
\Omega
=
\{\lambda\in\mathfrak{h}^*
  \mid v_{j,\lambda}\ne0 \text{ for some }j\in J\}.
\]
This is a finite nonempty set.
Choose a minimal element $\mu\in\Omega$ with respect to
$\preceq$, and let $v=v_{j_*,\mu}\ne0$ for some $j_*\in J$.

If $\mu$ is not the lowest weight $\delta_n-\delta_{n-r}$
of $V(\delta_r)$, then there exists $1\leq q\leq n-1$
such that $E_{q+1,q}v\ne0$.
Indeed, in the exterior-power realization of $V(\delta_r)$,
every weight space is spanned by an exterior basis vector.
Unless this vector is
$\varepsilon_{n-r+1}\wedge\cdots\wedge\varepsilon_n$,
its index set contains some $q$ but not $q+1$.
Thus $E_{q+1,q}$ acts nontrivially on it.
The vector $E_{q+1,q}v$ has weight
$\mu-\alpha_q\prec\mu$.

Set
\[
y'
=
(t_{q+1}\partial_q)\cdot y
=
\sum_{j\in J}
\bigl(
(t_{q+1}\partial_qp_j)\otimes v_j
+
p_j\otimes E_{q+1,q}v_j
\bigr)
\in N.
\]
Since $\mu$ is minimal in $\Omega$, we have
$\mu-\alpha_q\notin\Omega$.
Consequently, the component of $y'$ in
$P\otimes V(\delta_r)_{\mu-\alpha_q}$ is
\[
\sum_{j\in J}p_j\otimes E_{q+1,q}v_{j,\mu}.
\]
This component is nonzero because the $p_j$ are linearly
independent and
$E_{q+1,q}v_{j_*,\mu}=E_{q+1,q}v\ne0$.
In particular, $y'\ne0$.

Moreover, $\mu-\alpha_q$ is minimal among the weights
occurring in the second tensor factor of $y'$.
Indeed, every such weight has the form $\nu$ or
$\nu-\alpha_q$ for some $\nu\in\Omega$.
If either of these weights were strictly lower than
$\mu-\alpha_q$, then $\nu\prec\mu$, contradicting the
minimality of $\mu$.

We may rewrite $y'$ as a finite sum of pure tensors whose
first factors are linearly independent.
We then replace $y$ by $y'$, replace $\mu$ by
$\mu-\alpha_q$, and choose $v$ to be a nonzero component
of this new weight in one of the second tensor factors.
The preceding argument can therefore be repeated.

At each step, the selected weight strictly decreases.
Since $V(\delta_r)$ has only finitely many weights,
this process terminates after finitely many steps.
It can terminate only at the lowest weight
$\delta_n-\delta_{n-r}$, since otherwise the preceding
construction would give another strictly lower weight.
Thus the selected nonzero component $v$ satisfies
\[
v\in\mathbb{C}
\bigl(\varepsilon_{n-r+1}\wedge\cdots\wedge\varepsilon_n\bigr).
\]
Claim 1 follows.

Assume that $v$ is a nonzero weight component of some $v_{j_0}$, $j_0\in J$.
Then $E_{n-r,n-r+1}v_{j_0}\ne 0$ by Claim 1.

\textbf{Claim 2.}
There exists some $0\ne w_0\in V(\delta_{r-1})$ such that
$\pi_{r-1}(p\otimes w_0)\in N$
for all $p\in P$.

Since $y=\sum_{j\in J}p_j\otimes v_j\in N$
and $\iota(\partial_l)=\partial_l\otimes 1$ for all $1\leq l\leq n$,
we see that
$\sum_{j\in J}\partial_lp_j\otimes v_j\in N$
for all $1\leq l\leq n$. Hence, we have 
$\sum_{j\in J}\partial^\gamma p_j\otimes v_j\in N$
for any $\gamma\in\mathbb{Z}_+^n$.
For any $1\leq i\leq n-2$, by Lemma \ref{l6}, we have
\begin{equation*}
    \sum_{j\in J}{\sum_{s=1}^n{\partial _st^{\beta +e_{i}+e_{i+1}}
    \partial^\gamma p_j\otimes E_{s,i+2}E_{i,i+1}v_j}}\in N
\end{equation*}
for all $\beta,\gamma\in\mathbb{Z}_+^n$.
That is,
\begin{equation}\label{com2}
    \sum_{j\in J}{\sum_{s=1}^n{\partial _st^{e_{i}+e_{i+1}}
    zp_j\otimes E_{s,i+2}E_{i,i+1}v_j}}\in N
\end{equation}
for all $z\in D_n$.

Note that all $p_j$, $j\in J$, are linearly independent. 
By the density theorem in ring theory, for any $p\in P$, 
we can find some $z\in D_{n}$ such that
$zp_{j_0}=p$ and $zp_j=0$ for all $j\ne j_0$.
It follows from (\ref{com2}) that
$$\sum_{s=1}^n{\partial _st^{e_{i}+e_{i+1}}p\otimes E_{s,i+2}E_{i,i+1}v_{j_0}}\in N
$$
for all $p\in P$.
Since $n\geq 3$ and $2\leq r\leq n-1$, we have $1\leq n-r\leq n-2$.
Taking $i=n-r$, we get
\begin{equation}\label{eq57}
	\sum_{s=1}^n{\partial _st^{ e_{n-r}+e_{n-r+1}}p\otimes E_{s,n-r+2}E_{n-r,n-r+1}v_{j_0}}
\in N
\end{equation}
for all $p\in P$. We write $v_{j_0}=\varepsilon_{n-r+2}\land w+v_{j_0}^\prime$,
where $w\in \bigwedge^{r-1}V^\prime$, $v_{j_0}^\prime\in\bigwedge^rV^\prime$
and $V^\prime=
\mathrm{span}\{\varepsilon_1,\cdots,\varepsilon_{n-r+1},
\varepsilon_{n-r+3},\cdots,\varepsilon_n\}$.

By Claim 1, the lowest-weight component of $v_{j_0}$ is
\[
c\,\varepsilon_{n-r+1}\wedge\varepsilon_{n-r+2}
\wedge\cdots\wedge\varepsilon_n
\qquad\text{for some }c\in\mathbb{C}\setminus\{0\}.
\]
In the decomposition
\[
v_{j_0}=\varepsilon_{n-r+2}\wedge w+v_{j_0}^\prime,
\]
no exterior basis vector occurring in $v_{j_0}^\prime$
contains $\varepsilon_{n-r+2}$.
Hence the coefficient of
\[
\varepsilon_{n-r+1}\wedge
\varepsilon_{n-r+3}\wedge\cdots\wedge\varepsilon_n
\]
in $w$ is $-c$, since moving $\varepsilon_{n-r+2}$ to the
first position requires one interchange.

Moreover,
\[
E_{n-r,n-r+1}
\bigl(
\varepsilon_{n-r+1}\wedge
\varepsilon_{n-r+3}\wedge\cdots\wedge\varepsilon_n
\bigr)
=
\varepsilon_{n-r}\wedge
\varepsilon_{n-r+3}\wedge\cdots\wedge\varepsilon_n.
\]
In these expressions, the trailing exterior product is
omitted when $r=2$.
The action of $E_{n-r,n-r+1}$ replaces
$\varepsilon_{n-r+1}$ by $\varepsilon_{n-r}$.
Thus no other exterior basis vector of
$\bigwedge^{r-1}V^\prime$ can contribute to the coefficient
of the basis vector on the right-hand side.
Its coefficient in $E_{n-r,n-r+1}w$ is therefore
$-c\ne0$.

Set $w_0=E_{n-r,n-r+1}w$.
Then $w_0\ne0$, and
\begin{align*}
	&\sum_{s=1}^n{\partial _st^{ e_{n-r}+e_{n-r+1}}p\otimes E_{s,n-r+2}E_{n-r,n-r+1}v_{j_0}}\\
=&\sum_{s=1}^n{\partial _st^{ e_{n-r}+e_{n-r+1}}p\otimes 
E_{s,n-r+2}E_{n-r,n-r+1}(\varepsilon_{n-r+2}\land w)}\\
=&\sum_{s=1}^n{\partial _st^{ e_{n-r}+e_{n-r+1}}p\otimes 
(\varepsilon_{s}\land E_{n-r,n-r+1}w)}\\
=&\pi_{r-1}(t^{ e_{n-r}+e_{n-r+1}}p\otimes w_0)
\in N
\end{align*}
for all $p\in P$.

From
$$\partial _{n-r}\cdot\pi_{r-1} ( t^{e_{n-r}+e_{n-r+1}}p\otimes w_0  )
=\pi_{r-1}(t^{e_{n-r}+e_{n-r+1}}\partial _{n-r}p\otimes w_0)
+\pi_{r-1}(t^{e_{n-r+1}}p\otimes w_0)\in N,$$
we have $\pi_{r-1}(t^{e_{n-r+1}}p\otimes w_0)
\in N$ for all $p\in P$.
Similarly, from 
$\partial _{n-r+1}\cdot \pi_{r-1}\left( t^{e_{n-r+1}}p\otimes w_0 \right)\in N $,
we obtain that 
$\pi_{r-1}(p\otimes w_0)\in N$ for all $p\in P$.
Claim 2 follows.

Let $V:=\{w\in V(\delta_{r-1})|\pi_{r-1}(p\otimes w)\in N,\forall p\in P\}$
be a subspace of $V(\delta_{r-1})$.
From Claim 2, we see that
$V\ne 0$.

Take any $w\in V$, $p\in P$ and $m,k=1,2,\cdots,n$ with
$m\ne k$. We have
\begin{align*}
t_m\partial _k\cdot \pi_{r-1}\left( p\otimes w \right) 
=&\pi_{r-1}(t_m\partial _k\cdot(p\otimes w))\\
=&\pi_{r-1}(t_m\partial _kp\otimes w+p\otimes E_{mk}w)\\
=&\pi_{r-1}\left( t_m\partial _kp\otimes w \right) +\pi_{r-1}(p\otimes E_{mk}w)
\in N
.
\end{align*}
Thus $\pi_{r-1}(p\otimes E_{mk}w)\in N$ for any $p\in P$. Hence, $E_{mk}w\in V$.
This shows that $V$ is an $\mathfrak{sl}_n$-submodule of $V(\delta_{r-1})$, forcing
$V=V(\delta_{r-1})$. Then we obtain that
$L_n(P,r)=\pi_{r-1}(P\otimes V(\delta_{r-1}))\subseteq N$,
which implies that $N=L_n(P,r)$ and completes the proof of (a).

If $r\ne n-1$,
we have $F(P,V(\delta_r))/\widetilde{L_n}(P,r)\cong L_n(P,r+1)$,
which is simple by (a). Now (b) follows.

If $r=n-1$, we have 
$$F(P,V(\delta_{n-1}))/\widetilde{L_n}(P,n-1)\cong L_n(P,n)\cong\Delta P.$$
Now (c) and (d) follow from Proposition \ref{pro1}.
\end{proof}

Now we summarize the results obtained regarding $S_n$-modules 
$F(P,V(\delta_r))$, $0\leq r\leq n-1$,
as follows.
\begin{theorem}\label{t2}
	Let $P$ be a simple $D_n$-module. The following statements hold.
		
	(a) If $P\ncong A_n$, 
	then $F(P,V(\delta_0))=P$ is simple if and only if $\Delta P=P$.
	In non-simple cases, $F(P,V(\delta_0))=P$ has a unique
	simple submodule $\Delta P$ and the quotient $P/\Delta P$ is trivial.
	
	(b) $F(A_n,V(\delta_0))=A_n$ has a unique nonzero proper submodule $\mathbb{C}t^0$
	and thus has a unique simple quotient $A_n/\mathbb{C}t^0$.

	(c) $F(P,V(\delta_1))$ is not simple and
	it has a nonzero proper submodule $L_n(P,1)$.
	If $P\ncong A_n$, we have $L_n(P,1)\cong F(P,V(\delta_0))$.
	In addition, $L_n(A_n,1)\cong A_n/\mathbb{C}t^0$ is simple.
	
	(d) The quotient $F(P,V(\delta_1))/\widetilde{L_n}(P,1)\cong L_n(P,2)$
	is simple unless $n=2$ and $P\cong A_2$.
	In addition, $F(A_2,V(\delta_1))/\widetilde{L_2}(A_2,1)\cong A_2$ 
	has a unique simple quotient $A_2/\mathbb{C}t^0$.

	(e) For $n\geq 3$ and $2\leq r\leq n-1$,
	$F(P,V(\delta_r))$ is not simple and 
	it has a simple submodule $L_n(P,r)$.
	
	(f) For $n\geq 3$ and $2\leq r\leq n-2$,
	the quotient $F(P,V(\delta_r))/\widetilde{L_n}(P,r)\cong L_n(P,r+1)$ 
	is simple.

	(g) For $n\geq 3$,
	the quotient $F(P,V(\delta_{n-1}))/\widetilde{L_n}(P,{n-1})\cong \Delta P$
	is simple if $P\ncong A_n$.
	In addition, the quotient $F(A_n,V(\delta_{n-1}))/\widetilde{L_n}(A_n,n-1)\cong A_n$
	has a unique simple quotient $A_n/\mathbb{C}t^0$.
\end{theorem}
\begin{proof}
	(a) and (b) follow from Proposition \ref{pro1}. (c) follows from Proposition \ref{pro57}. 
	If $n>2$, the module $L_n(P,2)$ is simple by
	Proposition \ref{pro58}(a). If $n=2$, $L_2(P,2)\cong \Delta P$.
	Now (d) follows from Proposition \ref{pro1}.
	Finally, (e), (f) and (g) follow from Proposition \ref{pro58}.
\end{proof}

\section{Examples: weight modules}

In this section, 
we study the $S_n$-modules $F(P,M)$ in more detail under the assumption that both $P$ and $M$
are simple weight modules.
By Theorem \ref{t1}, if $M\ncong V(\delta_r)$ as an $\mathfrak{sl}_n$-module
for all $r=0,1,\cdots,n$,
we know that $F(P,M)$ is simple as an $S_n$-module. 
It remains to determine all nontrivial simple
$S_n$-subquotients of $F(P,V(\delta_r))$ for $0\leq r\leq n-1$.

A weight $W_n$-module is bounded if the dimensions of its weight spaces
are uniformly bounded by a constant positive integer.
Recall the following lemma from \cite{XL}.
\begin{lemma}[({\cite[Lemma 3.8]{XL}})]\label{l58}
	Let $P$ be a simple weight $D_n$-module
and $M$ be a simple weight $\mathfrak{gl}_n$-module.
Then $F(P,M)$ is a bounded $W_n$-module if and only if
$M$ is finite-dimensional.
\end{lemma}

Let $P$ be a simple weight $D_n$-module and let
$r\in\{0,1,\ldots,n\}$.
Since the $\mathfrak{gl}_n$-module $V(\delta_r,r)$ is
finite-dimensional, Lemma \ref{l58} implies that
$F(P,V(\delta_r,r))$ is a bounded $W_n$-module.
Set
\[
Q:=\widetilde{L_n}(P,r)/L_n(P,r).
\]
By the definition of $\widetilde{L_n}(P,r)$, we have
$W_nQ=0$.
In particular, each $d_i=t_i\partial_i$ acts as zero on $Q$.
Thus $Q=Q_0$, where the subscript $0$ denotes the zero-weight
space with respect to $d_1,\ldots,d_n$.

Since $L_n(P,r)$ and $\widetilde{L_n}(P,r)$ are
$W_n$-submodules of the weight module $F(P,V(\delta_r,r))$,
taking zero-weight spaces gives
\[
Q_0
\cong
\widetilde{L_n}(P,r)_0/L_n(P,r)_0.
\]
Consequently,
\[
\dim Q
=
\dim Q_0
\leq
\dim F(P,V(\delta_r,r))_0
<
\infty.
\]
Therefore, $Q$ is a finite-dimensional trivial $W_n$-module,
and hence also a finite-dimensional trivial $S_n$-module.
We will use this fact in the following discussion.

Recall that the full Fourier transform $F$ is the automorphism of $D_n$
determined by $F(t_i)=\partial_i$ and $F(\partial_i)=-t_i$
for $i=1,\ldots,n$.
Here and below, $A_n^F$ denotes the twist of the natural $D_n$-module
$A_n$ by $F$.
Thus, $A_n^F$ has the same underlying vector space as $A_n$, with
the action given by $x\cdot p=F(x)p$ for $x\in D_n$ and $p\in A_n$,
where the right-hand side uses the usual action of $D_n$ on $A_n$.

\begin{proposition}\label{p1}
	Let $P$ be a simple weight $D_n$-module. The following statements hold.

	(a) $F(P,V(\delta_0))=P$ is simple, 
	where $P\ncong A_n$ and $P\ncong A_n^F$.

	(b) $F(A_n,V(\delta_0))=A_n$ has a unique nontrivial irreducible subquotient
	$A_n/\mathbb{C}t^0$.

	(c) $F(A_n^F,V(\delta_0))=A_n^F$ has a unique nontrivial irreducible subquotient
	$\Delta F(A_n^F,V(\delta_0))=\Delta A_n^F$.
\end{proposition}
\begin{proof}
	By Lemma \ref{l2},
	$\Delta P=P$ if and only if
	$P\ncong A_n^F$. 
	Now the statements follow
	from Theorem \ref{t2}(a)(b).
\end{proof}

\begin{proposition}\label{p2}
	Let $P$ be a simple weight $D_n$-module.
	The following
statements hold.

	(a) For $n\geq 3$, $P\ncong A_n$ and $P\ncong A_n^F$, 
	the nontrivial irreducible subquotients of $F(P,V(\delta_1))$
	are $F(P,V(\delta_0))=P$ and $L_n(P,2)$.

	(b) For $n\geq 3$, the nontrivial irreducible subquotients of $F(A_n,V(\delta_1))$
	are $A_n/\mathbb{C}t^0$ and $L_n(A_n,2)$.

	(c) For $n\geq 3$, the nontrivial irreducible subquotients of $F(A_n^F,V(\delta_1))$
	are $\Delta A_n^F$ and $L_n(A_n^F,2)$.

	(d) For $n=2$, $P\ncong A_2$ and $P\ncong A_2^F$,
	 $F(P,V(\delta_1))$ has a unique 
	nontrivial irreducible subquotient
	$F(P,V(\delta_0))=P$.

	(e) $F(A_2,V(\delta_1))$ has a unique nontrivial irreducible subquotient
	$A_2/\mathbb{C}t^0$.

	(f) $F(A_2^F,V(\delta_1))$ has a unique nontrivial irreducible subquotient
	$\Delta A_2^F$.
\end{proposition}
\begin{proof}
	Consider the following chain of submodules:
	$$0\subseteq L_n(P,1)\subseteq \widetilde{L_n}(P,1)
	\subseteq F(P,V(\delta_1)). $$
The quotient $\widetilde{L_n}(P,1)/L_n(P,1)$ is a finite-dimensional trivial
	module.
From Theorem \ref{t2}(c), 
$L_n(P,1)$ is isomorphic to $F(P,V(\delta_0))$ if $P\ncong A_n$, and to $A_n/\mathbb{C}t^0$ if $P\cong A_n$.
From Theorem \ref{t2}(d), 
$$F(P,V(\delta_1))/\widetilde{L_n}(P,1)\cong L_n(P,2)$$ is simple unless
$n=2$ and $P\cong A_2$.

In addition, $F(A_2,V(\delta_1))/\widetilde{L_2}(A_2,1)\cong A_2$.
Hence, there exists an $S_2$-submodule $N$ of $F(A_2,V(\delta_1))$ such that
$\widetilde{L_2}(A_2,1)\subseteq N\subseteq F(A_2,V(\delta_1))$,
where $N/\widetilde{L_2}(A_2,1)\cong \mathbb{C}t^0$ is trivial
and $F(A_2,V(\delta_1))/N\cong A_2/\mathbb{C}t^0$ is simple.

Now the proposition follows from Proposition \ref{p1}.
\end{proof}

\begin{proposition}\label{p3}
	Let $n\geq 3$, $2\leq r\leq n-1$ and $P$ be a simple weight $D_n$-module.
	The following statements hold.
	
	(a) If $r\ne n-1$, the nontrivial irreducible subquotients of $F(P,V(\delta_r))$
	are $L_n(P,r)$ and $L_n(P,r+1)$.
	
	(b) If $P\ncong A_n$, the nontrivial irreducible subquotients of $F(P,V(\delta_{n-1}))$
	are $L_n(P,n-1)$ and $\Delta P$.

	(c) The nontrivial irreducible subquotients of $F(A_n,V(\delta_{n-1}))$
	are $L_n(A_n,n-1)$ and $A_n/\mathbb{C}t^0$.
\end{proposition}
\begin{proof}
	Consider the following chain of submodules:
	$$0\subseteq L_n(P,r)\subseteq \widetilde{L_n}(P,r)\subseteq F(P,V(\delta_r)).$$

	The quotient $\widetilde{L_n}(P,r)/L_n(P,r)$ is a finite-dimensional trivial module.
	By Theorem \ref{t2}(e), $L_n(P,r)$ is simple. 
	From Theorem \ref{t2}(f)(g), the quotient $F(P,V(\delta_r))/\widetilde{L_n}(P,r)\cong L_n(P,r+1)$ 
	is simple unless $P\cong A_n$ and $r=n-1$.

	Moreover, by Theorem \ref{t2}(g), 
	$F(A_n,V(\delta_{n-1}))/\widetilde{L_n}(A_n,n-1)\cong \Delta A_n=A_n$
	has a unique nonzero proper submodule $\mathbb{C}t^0$.
	Hence, there exists some $S_n$-submodule $N$ of $F(A_n,V(\delta_{n-1}))$
	such that
	$\widetilde{L_n}(A_n,n-1)\subseteq N\subseteq F(A_n,V(\delta_{n-1}))$,
	where $N/\widetilde{L_n}(A_n,n-1)\cong \mathbb{C}t^0$ is 
	trivial and
	$F(A_n,V(\delta_{n-1}))/N\cong A_n/\mathbb{C}t^0$ is simple.

	Now the proposition follows.
\end{proof}

\vspace*{2em}
\noindent\textbf{\Large Declarations}
\vspace*{1em}

\begin{minipage}{0.9\textwidth}
	\llrrlr[\textbf{Ethical approval\quad}][Not applicable.
	]

	\llrrlr[\textbf{Funding\quad}][R. L\"u is partially 
supported by the National Natural Science Foundation of China (Grant No. 12271383).
	]

	\llrrlr[\textbf{Availability of data and materials\quad}][No datasets were generated or analysed during the current study.]

	\llrrlr[\textbf{Conflicts of interest\quad}][The authors have no conflicts of interest to declare that are relevant to this article.
	]

	\llrrlr[\textbf{Competing interests\quad}][The authors declare no competing interests.
	]

\end{minipage}

\phantomsection

\end{document}